\documentclass[12pt]{amsart}  %12

\usepackage{fullpage}
\usepackage{url,amssymb,amsmath,amsfonts,amsthm,mathrsfs}

\DeclareFontEncoding{OT2}{}{} % to enable usage of cyrillic fonts
\newcommand{\textcyr}[1]{%
 {\fontencoding{OT2}\fontfamily{cmr}\fontseries{m}\fontshape{n}\selectfont #1}}
\newcommand{\Sha}{{\mbox{\textcyr{Sh}}}}

\usepackage[all]{xy} \SelectTips{cm}{10}

\usepackage{color}

\definecolor{webcolor}{rgb}{0.8,0,0.2}
\definecolor{webbrown}{rgb}{.6,0,0}
\usepackage[
        colorlinks,
        linkcolor=webbrown,  filecolor=webcolor,  citecolor=webbrown, 
        backref,
        pdfauthor={David Zywina}, 
       pdftitle={},
]{hyperref}
\usepackage[alphabetic,backrefs,lite]{amsrefs} % for bibliography

\numberwithin{equation}{section}

\newcommand{\FF}{\mathbb F}

\newcommand{\PP}{\mathbb P}
\newcommand{\QQ}{\mathbb Q}

\newcommand{\ZZ}{\mathbb Z}

\newcommand{\OO}{\mathcal O}

\newcommand{\calB}{\mathcal B}

\newcommand{\calA}{\mathcal A}

\newcommand{\calN}{\mathcal N}

\newcommand{\calM}{\mathcal M}

\newcommand{\p}{\mathfrak p}
\newcommand{\m}{\mathfrak m}

\newcommand{\scrE}{\mathscr E}

\newcommand{\scrU}{\mathscr U}

\newcommand{\scrF}{\mathscr F}

\def\Im{{\operatorname{Im}}}

\def\rank{{\operatorname{rank}}}

\def\Spec{\operatorname{Spec}} 
 
\def\Sel{\operatorname{Sel}} 

\def\Gal{\operatorname{Gal}}

\newcommand{\q}{\mathfrak q}

\newcommand{\legendre}[2]{\genfrac{(}{)}{}{}{#1}{#2}}

\newcommand{\defi}[1]{\textsf{#1}} % for defined terms

\newcommand\blank[1]{}

\def\bbar#1{\setbox0=\hbox{$#1$}\dimen0=.2\ht0 \kern\dimen0 
\overline{\kern-\dimen0 #1}}
 
\newcommand{\Kbar}{{\bbar{K}}}

\newtheorem{thm}{Theorem}[section]
\newtheorem{lemma}[thm]{Lemma}
\newtheorem{cor}[thm]{Corollary}
\newtheorem{prop}[thm]{Proposition}

\theoremstyle{definition}

\theoremstyle{remark}

\newenvironment{romanenum}{\hfill \begin{enumerate} }{\end{enumerate}}
\newenvironment{alphenum}{\hfill \begin{enumerate} }{\end{enumerate}}

\begin{document}

\title{On the infinitude of elliptic curves over a number field with prescribed small rank}

% 11G05 Elliptic curves over global fields
% 14J27 Elliptic surfaces, elliptic or Calabi-Yau fibrations
% \keywords{}
\author{David Zywina}
\address{Department of Mathematics, Cornell University, Ithaca, NY 14853, USA}
\email{zywina@math.cornell.edu}

\begin{abstract}
For any number field $K$ and integer $0\leq r \leq 4$, we prove that there are infinitely many elliptic curves over $K$ of rank $r$.   Our elliptic curves are obtained by specializing well-chosen nonisotrivial elliptic curves over the function field $K(T)$.   We use a result of Kai, which generalizes work of Green, Tao and Ziegler to number fields, to choose our specializations so that we have control over the bad primes and can perform a $2$-descent to compute ranks.
\end{abstract}

\maketitle

\section{Introduction}

Let $K$ be a number field.  For an elliptic curve $E$ defined over $K$, the abelian group $E(K)$ consisting of the $K$-rational points of $E$ is finitely generated.   The \defi{rank} of $E$ is the rank of the abelian group $E(K)$.   

There are two natural questions concerning the distribution of ranks.  What integers occur as the rank of an elliptic curve over $K$?   What integers occur as the rank of infinitely many elliptic curves over $K$?   Our main theorem addresses these questions for small ranks.  Let $\Kbar$ be an algebraic closure of $K$.

\begin{thm} \label{T:main}
For any number field $K$ and integer $0\leq r \leq 4$, there are infinitely many elliptic curves over $K$, up to isomorphism over $\Kbar$, of rank $r$.
\end{thm}

This theorem covers all cases of number fields $K$ and integers $r$ for which it was previously known that there are infinitely many elliptic curves over $K$ of rank $r$. 
See \S\ref{SS:earlier} for some previous results.  

We state the following immediate consequence to emphasize that even the existence of an elliptic curve of a given rank is a nontrivial result when working over a general number field.

\begin{cor} \label{C:main}
For any number field $K$ and integer $0\leq r \leq 4$, there is an elliptic curve $E$ over $K$ of rank $r$.
\end{cor}

Specialization gives a natural method for constructing many elliptic curves of large rank.  From Elkies \cite[II]{Elkies}, there is a nonisotrivial elliptic curve $E$ over the function field $\QQ(T)$ so that the finitely generated group $E(\QQ(T))$ has rank $18$.  A theorem of Silverman \cite{MR703488} then implies that for all but finitely many $t\in K$, specializing a fixed Weierstrass model of $E$ at $t$ produces an elliptic curve $E_t$ over $K$ of rank \emph{at least} $18$.  Unfortunately, it is unclear how to find a fixed $r\geq 18$ such that $E_t$ has rank $r$ for infinitely many $t\in K$.   For each $0\leq r\leq 4$, our theorem is proved by using a specific nonisotrivial elliptic curve $E/\QQ(T)$ of rank $r$ and showing that there are infinitely many $t\in K$ for which $E_t/K$ has rank $r$.

Concerning what ranks actually occur over a fixed number field, it is not even clear if they are uniformly bounded, cf.~\cite[\S3]{heuristic} for a brief history of the problem.  In \cite{heuristic}, a heuristic is given for uniformly bounded ranks which predicts that there are infinitely many elliptic curves over $\QQ$ of rank $r$ for each $0\leq r\leq 20$ and that there are only finitely many elliptic curves over $\QQ$ of rank greater than $21$.

\subsection{Earlier results} \label{SS:earlier}

The $r=0$ case of Theorem~\ref{T:main} was proved by Mazur and Rubin, cf.~\cite[Theorem 1.11]{MR2660452}.   The $r=1$ case of Theorem~\ref{T:main} was until recently known only for a few number fields $K$;  for example,  see \cite[Th\'eor\`eme~3.1]{MR870738} or \cite{MR3237733} for $K=\QQ$.

Other special cases of Theorem~\ref{T:main} are quite recent.  With $K=\QQ$, the $r=2$ case of Theorem~\ref{T:main} was proved in \cite{Zyw25b}.  The full $r=1$ case was proved in \cite{Zyw25c} and also independently by Koymans and Pagano \cite{KPnew} who built off their ideas in \cite{KP}. In \cite{Savoie}, it is shown that there are infinitely many elliptic curves over $\QQ(i)$ of rank $2$; however, they all have $j$-invariant $1728$.  

The basic strategy in the author's recent rank papers \cite{Zyw25a,Zyw25b,Zyw25c} is to consider a well-chosen nonisotrivial $E$ over a function field $K(T)$ that has rank $r$ and to prove that the specialization $E_t$ over $K$ also has rank $r$ for infinitely many $t\in K$.  In \cite{KPnew}, rank $1$ elliptic curves over $K$ are constructed by considering quadratic twists of a ``generic'' elliptic over $K$ with full $2$-torsion.

\subsection{Overview}

We give some motivation for our approach in the optional section \S\ref{S:ideas}.  In \S\ref{S:background}, we recall what we need on Selmer groups and compute some local conditions.   In \S\ref{S:general}, we state an axiomatic version of our theorem which is proved in \S\ref{S:proof of general}.   Finally in \S\ref{S:proof of main}, we use five explicit elliptic curves $E/K(T)$ to prove Theorem~\ref{T:main} for each $0\leq r\leq 4$ by applying our axiomatic version.

\subsection{Notation}

Let $K$ be a number field and let $\OO_K$ be its ring of integers.   For each nonzero prime ideal $\p$ of $\OO_K$, let $v_\p$ be the discrete valuation on $K$ normalized so that $v_\p(K^\times)=\ZZ$.   

For a finite set $S$ of nonzero prime ideals of $\OO_K$, let $\OO_{K,S}$ be the ring of $S$-integers, i.e., the ring of $a\in K$ for which $v_\p(a)\geq 0$ for all nonzero prime ideals $\p\notin S$ of $\OO_K$.

For a place $v$ of $K$, we denote by $K_v$ the completion of $K$ at $v$.    When $v$ is finite, we let $\OO_v$ be the valuation subring of $K_v$.  Note that when $v$ is a finite place of $K$, we will frequently switch between $v$ and the corresponding prime ideal $\p$ of $\OO_K$.

\section{Ideas and motivation} \label{S:ideas}

We now describe some strategy and motivation for our proof; this will not be referred to later and can be safely skipped.  Let $K$ be a number field.   We will identify $K(T)$ with the function field of $\PP^1_K =\Spec K[T] \cup \{\infty\}$.

Consider a nonisotrivial elliptic curve $E$ over $K(T)$  defined by a Weierstrass of the form
\[
y^2=x^3+a x^2 +b x,
\]
with $a,b\in K[T]$, that is minimal over $K[T]$.   The point $(0,0)\in E(K(T))$ has order $2$.  There is a degree $2$ isogeny $\phi\colon E\to E'$ whose kernel is generated by $(0,0)$ and we may take $E'$ to be defined the elliptic curve over $K(T)$ defined by 
\[
y^2=x^3+a'x^2+b'x
\]
with $a':=-2a$ and $b':=a^2-4b$.    We assume that the groups $E(K(T))$ and $E'(K(T))$ are known.

Let $\pi\colon \scrE \to \PP^1_K$ be the smooth minimal elliptic surface corresponding to $E/K(T)$.   Let $\calB$ be the (finite) set of points in $\PP^1(K)$ over which the fiber of $\pi$ is singular.    After replacing $T$ by another generator of the function field of $\PP^1_K$ over $K$ and changing the Weierstrass model, we may assume that $\calB$ is a subset of $\OO_K \subseteq K\cup \{\infty\} = \PP^1(K)$.

 Let $\Delta\in K[T]$ be the discriminant of our Weierstrass model of $E$.  Note that $\calB$ is the set of roots of $\Delta$ in $K$ since the model is minimal.  For each $t\in K-\calB$, specializing our models at $t$ gives an elliptic curve $E_t$ over $K$ defined by $y^2=x^3+a(t) x^2 +b(t) x$ and an elliptic curve $E'_t$ over $K$ defined by $y^2=x^3+a'(t) x^2 +b'(t) x$.  Using that the fiber of $\pi$ over $\infty$ is smooth, one can show that there is a finite set $S$ of nonzero prime ideals of $\OO_K$ such that for each $t\in K-\calB$, $E_t$ has good reduction at a nonzero prime ideal $\p\notin S$ of $\OO_K$ if and only if $\Delta(t)$ has positive $\p$-adic valuation.

The abelian group $E(K(T))$ is finitely generated and we denote its rank by $r$. From Silverman \cite{MR703488}, specialization by $t$ defines an injective homomorphism $E(K(T))\to E_t(K)$ for all but finitely many $t\in K-\calB$.  In particular, we have 
 \[
  \rank \, E_t(K) \geq r
 \] 
 for all but finitely many $t\in K-\calB$.     The challenge is to prove the existence of $t\in K-\calB$ for which $E(K)$ has rank \emph{exactly} $r$.  A \emph{minimalist conjecture} predicts that the density of $t\in K-\calB$ for which $E_t$ has rank $r$ will be positive (the heuristic being that a typical curve over $K$ in the family should have the smallest possible rank that is compatible with the parity conjecture).  

The main approach to computing upper bounds of ranks is via Selmer groups.  In our case, we will perform descent by $2$-isogeny, cf.~\S\ref{SS:2-descent} for details.   Take any $t\in K-\calB$.  There is an isogeny $\phi_t\colon E_t \to E_t'$ of degree $2$ whose kernel is generated by $(0,0)$.  Using group cohomology, we obtain a group homomorphism
\[
\delta_{E_t'}\colon E'_t(K) \to K^\times/(K^\times)^2
\]
with kernel $\phi_t(E(K))$ that satisfies $\delta_{E'_t}((x,y))=x\cdot (K^\times)^2$ for $(x,y)\in E_t'(K)-\{0,(0,0)\}$ and $\delta_{E'_t}((0,0))= b'(t) \cdot (K^\times)^2$.  Similarly for each place $v$ of $K$, we have a homomorphism $\delta_{E_t',v}\colon E'_t(K_v) \to K_v^\times/(K_v^\times)^2$.   

We can define the $\phi_t$-\emph{Selmer group} $\Sel_{\phi_t}(E_t/K)$ to be the subgroup of $K^\times/(K^\times)^2$ consisting of those cosets whose image in $K_v^\times/(K_v^\times)^2$ lies in the image of $\delta_{E_t',v}$ for all places $v$ of $K$.  The group $\Sel_{\phi_t}(E_t/K)$ is finite and $\delta_{E'_t}$ induces an injective homomorphism
\[
E'_t(K)/\phi_t(E_t(K)) \hookrightarrow \Sel_{\phi_t}(E_t/K).
\]
Similarly, we have a group homomorphism $\delta_{E_t}\colon E_t(K) \to K^\times/(K^\times)^2$ that induces an injective homomorphism
\[
E_t(K)/\hat\phi_t(E'_t(K)) \hookrightarrow \Sel_{\hat\phi_t}(E'_t/K),
\]
where $\hat\phi_t\colon E_t' \to E_t$ is the dual isogeny of $\phi_t$.  Viewing our groups as vector spaces over $\FF_2$, we have 
\[
\rank \, E_t(K) =\dim_{\FF_2} \delta_{E_t}(E_t(K)) + \dim_{\FF_2} \delta_{E_t'}(E_t'(K)) - 2.
\]
cf.~Lemma~\ref{L:rank via quotients}, which gives the upper bound
\[
\rank \, E_t(K) \leq  \dim_{\FF_2} \Sel_{\hat\phi_t}(E'_t/K) + \dim_{\FF_2} \Sel_{\phi_t}(E_t/K) - 2.
\]
The Selmer groups $\Sel_{\hat\phi_t}(E'_t/K)$ and $\Sel_{\phi_t}(E_t/K)$ are in principle computable and hence we have a computable upper bound for the rank of $E_t$.   Unfortunately, the computation of Selmer groups requires knowledge of the primes of bad reduction of $E_t$ which is difficult to control if we are varying over infinitely many $t$.

To control the primes of bad reduction, we will make the additional assumption that $\calB$ is also the the set of points in $\PP^1(\Kbar)$ over which the fiber of $\pi$ is singular.  Equivalently, $\Delta$ factors into linear polynomials in $K[T]$.  

Let $\OO_{K,S}$ be the ring of $S$-integers in $K$ and let $S_\infty$ be the set of archimedean places of $K$.   For each $v\in S$, fix a nonempty open subset $U_v$ of $K_v^2$.  For each place $v\in S_\infty$, fix a nonempty open subset $U_v$ of $K_v$.    We will make use of a result of Kai \cite[Proposition~13.2]{Kai} which implies that there are nonzero $m,n\in \OO_{K,S}$ such that the following hold:
\begin{itemize}
\item
$m-en$ with $e\in \calB$ generate distinct prime ideals of $\OO_{K,S}$,
\item 
$(m,n)$ lies in $U_v$ for all $v\in S$,
\item 
$m/n$ lies in $U_v$ for all $v\in S_\infty$.
\end{itemize}
Kai's theorem is a generalization of work of Green, Tao and Ziegler to number fields.  After possibly increasing $S$ first and taking $t:=m/n\in K$ with $m$ and $n$ as above, $E_t$ has bad reduction at a nonzero prime ideal $\p\notin S$ of $\OO_K$ if and only if $\p\OO_{K,S}=(m-en)\OO_{K,S}$ for some $e\in \calB$.  The type of reduction that $E_t$ has at a prime ideal $\p\in S$ can be forced by making suitable choice of $U_\p$.    By shrinking the sets $U_v$, we may also assume that $\rank\, E_t(K) \geq r$ always holds when $t:=m/n$.

The goal is to possible extend the finite set $S$ and choose sets $\{U_v\}_{v\in S\cup S_\infty}$ appropriately so that, with $m$ and $n$ as above and $t:=m/n$, the local conditions imposed guarantee that the homomorphisms
\begin{align} \label{E:specialization mot1}
E(K(T))\to E_t(K) \xrightarrow{\delta_{E_t}} \Sel_{\hat\phi_t}(E'_t/K)
\end{align}
and
\begin{align}
\label{E:specialization mot2}
E'(K(T))\to E'_t(K) \xrightarrow{\delta_{E'_t}} \Sel_{\phi_t}(E_t/K)
\end{align}
are both surjective.   If (\ref{E:specialization mot1}) and (\ref{E:specialization mot2}) were both surjective, then we would know the image of $\delta_{E_t}$ and $\delta_{E'_t}$ from which we could deduce the upper bound $\rank\, E_t(K)\leq r$ and thus prove that $\rank\, E_t(K)=r$.

However, we still might have $\rank\, E_t(K)= r$ even if one of the homomorphisms (\ref{E:specialization mot1}) or (\ref{E:specialization mot2}) is not surjective;  in this case one can show that $\Sha(E_t/K)[2]\neq 0$ or $\Sha(E'_t/K)[2]\neq 0$.  In \S\ref{S:general}, we give additional conditions on $E$ to ensure this surjectivity and in particular avoid $2$-torsion in our Tate--Shafarevich groups; many of the conditions are constraints on the singular fibers of $\scrE\to \PP^1_K$.   

Let $\calN$ be the \emph{conductor} of the elliptic surface $\pi\colon \scrE\to \PP^1_K$, i.e., the divisor of $\PP_K^1$ supported on the closed points of $\PP^1_K$ over which the fiber of $\pi$ is singular and each such point has multiplicity $1$ if the fiber is multiplicative and multiplicity $2$ if the fiber is additive.  There is a simple upper bound
\[
r=\rank \, E(K(T)) \leq \rank \, E(\Kbar(T)) \leq \deg \calN -4,
\]
cf.~\cite[Corollary~2]{MR1211006}.  In \S\ref{S:general}, we will impose conditions that imply $\rank \, E(K(T))=\deg \calN-4$.  Having ``maximal rank'' will make the proof simpler and in particular make it easier to find $t=m/n$ for which (\ref{E:specialization mot1}) and (\ref{E:specialization mot2}) are surjective.  This is strong constraint on $E$ and it implies that the surface $\scrE$ has geometric genus $0$, cf.~\cite[equation (1.13$'$)]{MR1211006}.  For the explicit $E/K(T)$ we will consider for our application, $\scrE$ will be a rational elliptic surface with a $2$-torsion point and hence $r\leq 4$ by \cite[Corollary~2.1]{MR1104782}.\\

In our proof of Theorem~\ref{T:main}, we will use an explicit elliptic curve $E/K(T)$ for each integer $0\leq r \leq 4$.   For example when $r=4$, we will use the elliptic curve $E$ over $K(T)$ defined by the Weierstrass equation
\[
y^2=x^3-70(T^2-25^2)\cdot x^2+2^4 7^2 (T^2-11^2)(T^2-25^2) \cdot x
\]
whose discriminant is 
\[
\Delta=2^{14}3^2 7^6 (T-11)^2(T+11)^2  (T-25)^3(T+25)^3 (T-39)(T+39).
\]
The set of points of $\PP^1_K$ over which the corresponding elliptic surface $\scrE\to \PP^1_K$ is singular is $\calB:=\{\pm 11, \pm 25, \pm 39\}$.  The conductor of $E$ is $\calN= (-11)+(11)+(-39)+(39)+2\cdot (-25) + 2\cdot (25)$ and hence $\rank\, E(K(T))\leq \deg \calN-4=4$.   In fact, $E(K(T))$ has rank $4$.  Moreover, $E(K(T))$ is isomorphic to $\ZZ/2\ZZ\times \ZZ^4$ and is generated by the $2$-torsion point $(0,0)$ and the following independent points:
\begin{align*}
&(14(T-11)(T+11), 1176(T-11)(T+11)),\\
&(2(T-11)(T-25), 12(3T+65)(T-11)(T-25)),\\
&(14(T-25)(T+25) , 2352(T-25)(T+25)),\\
&(8(T+11)(T+25) , -48(T+11)(T+25)(T-45)).
\end{align*}   
Our proof will show that the elliptic curve $E_t$ over $K$ has rank $4$ for infinitely many $t\in K-\calB$.

\section{Background on descent}
\label{S:background}

We recall some background on descent by a $2$-isogeny.  See \cite[\S X.4]{Silverman}, and in particular \cite[\S X.4 Example 4.8]{Silverman}, for the basics.

\subsection{$2$-descent} \label{SS:2-descent}

Fix a field $K$ of characteristic $0$ and let $\Kbar$ be an algebraic closure.   

Let $E$ be an elliptic curve over $K$ with a fixed $K$-rational point $P_0$ of order $2$.    There is a degree $2$ isogeny $\phi\colon E\to E'$ whose kernel is generated by $P_0$. Let $\hat{\phi}\colon E'\to E$ be the dual isogeny of $\phi$.  We can choose a Weierstrass model for $E$ of the form
\begin{align} \label{E:basic model}
y^2=x^3+a x^2+b x
\end{align}
with $a,b\in K$ and $P_0=(0,0)$.    We may take $E'$ to be the elliptic curve over $K$ defined by the Weierstrass equation $y^2=x^3+a'x^2+b'x$, where $a':=-2a$ and $b':=a^2-4b$, and take $\phi\colon E\to E'$ to be
\[
\phi(x,y)=(y^2/x^2, y(b-x^2)/x^2).
\]  
The kernel of $\hat{\phi}$ is generated by the $2$-torsion point $(0,0)$ of  $E'$.

Set $\Gal_K:=\Gal(\Kbar/K)$.  Starting with the short exact sequence $0\to \ker \phi \to E \xrightarrow{\phi} E'\to 0$ and taking Galois cohomology yields an exact sequence
\[
0 \to \ker \phi \to E(K)\xrightarrow{\phi} E'(K) \xrightarrow{\delta_{E'}} H^1(\Gal_K, \ker \phi).
\]
Since $\ker \phi$ and $\{\pm 1\}$ are isomorphic $\Gal_K$-modules, we have a natural isomorphism
\begin{align} \label{E:H1 isom}
H^1(\Gal_K, \ker \phi )\xrightarrow{\sim} H^1(\Gal_K, \{\pm 1\}) \xrightarrow{\sim} K^\times/(K^\times)^2.
\end{align}
Using the isomorphism (\ref{E:H1 isom}), we may view $\delta_{E'}$ as a homomorphism 
\[
\delta_{E'}\colon E'(K)\to K^\times/(K^\times)^2
\]   
whose kernel is $\phi(E(K))$.  In particular, we can identify $E'(K)/\phi(E(K))$ with a subgroup of $K^\times/(K^\times)^2$.  For any point $(x,y)\in E'(K)-\{0,(0,0)\}$, we have 
\[
\delta_{E'}((x,y))=x\cdot (K^\times)^2.
\] 
We also have $\delta_{E'}(0)=1$ and $\delta_{E'}((0,0))=b' \cdot (K^\times)^2$.   

In the same manner, we obtain a homomorphism $\delta_E\colon E(K)\to K^\times/(K^\times)^2$ with kernel $\hat\phi(E'(K))$ that satisfies $\delta_{E}((x,y))=x\cdot (K^\times)^2$ for $(x,y)\in E(K)-\{0,(0,0)\}$ and $\delta_E((0,0))=b \cdot (K^\times)^2$.  We now show that the images of $\delta_E$ and $\delta_{E'}$ determine the rank of $E$.

\begin{lemma} \label{L:rank via quotients}
If $E(K)$ is a finitely generated abelian group of rank $r$, then 
\[
r=\dim_{\FF_2} \delta_E(E(K)) + \dim_{\FF_2} \delta_{E'}(E'(K)) - 2.
\]
\end{lemma}
\begin{proof}
Since $E(K)$ is a finitely generated abelian group of rank $r$, the group $E(K)/2E(K)$ is finite and we have $\dim_{\FF_2} E(K)/2E(K) = r + \dim_{\FF_2} E(K)[2]$.  There is an exact sequence 
\begin{align} \label{E:exact MW last}
0\to \frac{\langle (0,0) \rangle}{\phi(E(K)[2])} \to \frac{E'(K)}{\phi(E(K))} \xrightarrow{\hat\phi} \frac{E(K)}{2E(K)} \to \frac{E(K)}{\hat{\phi}(E'(K))} \to 0.
\end{align}
We have $\dim_{\FF_2} {\langle (0,0) \rangle}/{\phi(E(K)[2])}= 1 - (\dim_{\FF_2} E(K)[2]-1)=2-\dim_{\FF_2} E(K)[2]$ and hence from (\ref{E:exact MW last}) we find that $\dim_{\FF_2} E(K)/2E(K)$ is equal to
\[
\dim_{\FF_2} {E(K)}/{\hat\phi(E'(K))} + \dim_{\FF_2} {E'(K)}/{\phi(E(K))} - (2-\dim_{\FF_2} E(K)[2]).
\]
The lemma follows by combining our two expressions for $\dim_{\FF_2} E(K)/2E(K)$ and using the isomorphisms $\delta_{E}(E(K))\cong{E(K)}/{\hat\phi(E'(K))}$ and $\delta_{E'}(E'(K))\cong{E'(K)}/{\phi(E(K))}$.
\end{proof}

\subsection{Number field case}
Now assume further that $K$ is a number field.  Take any place $v$ of $K$.  Applying the construction of \S\ref{SS:2-descent} with $E$ base changed to $K_v$ we obtain a homomorphism
\[
\delta_{E',v} \colon E'(K_v) \to K_v^\times/(K_v^\times)^2
\]
with kernel $\phi(E(K_v))$.   We can identify the \defi{$\phi$-Selmer group} of $E/K$, which we denote by $\Sel_\phi(E/K)$, with the subgroup of $K^\times/(K^\times)^2$ consisting of those square classes whose image in $K_v^\times/(K_v^\times)^2$ lies in $\operatorname{Im}(\delta_{E',v})$ for all places $v$ of $K$.  The image of $\delta_{E'}$ lies in $\Sel_\phi(E/K)$ and hence we have an injective homomorphism
\[
E'(K)/\phi(E(K)) \hookrightarrow \Sel_\phi(E/K).
\]
Similar, we have an injective homomorphism $E(K)/\hat\phi(E'(K)) \hookrightarrow \Sel_{\hat\phi}(E'/K)$.   These Selmer groups are finite and  are in principle computable.  Using Lemma~\ref{L:rank via quotients}, these Selmer groups give an upper bound on the rank of $E$.  The following lemma links their cardinalities.

\begin{lemma} \label{L:Selmer ratio}
\begin{romanenum}
\item \label{L:Selmer ratio i}
We have 
\[
|\Sel_\phi(E/K)|/|\Sel_{\hat\phi}(E'/K)| = \prod_v \tfrac{1}{2} |\operatorname{Im}(\delta_{E',v})|,
\] 
where the product is over the places $v$ of $K$.
\item \label{L:Selmer ratio ii}
Let $\p$ be a nonzero prime ideal of $\OO_K$ that does not divide $2$.   Then 
\[
\tfrac{1}{2} |\operatorname{Im}(\delta_{E',\p})| = c_\p(E')/c_\p(E),
\] 
where $c_\p(E)$ and $c_\p(E')$ are the Tamagawa numbers of $E$ and $E'$, respectively, at $\p$. 
\end{romanenum}
\end{lemma}
\begin{proof}
This was shown in \cite[Lemma~2.1]{Zyw25c}.  Part (\ref{L:Selmer ratio i}) was a consequence of a result of Cassels \cite{MR179169} and part (\ref{L:Selmer ratio ii}) was deduced from \cite{MR3324930}.
\end{proof}

\subsection{Some local computations} \label{SS:function field case}

Let $R$ be a complete discrete valuation ring and let $K$ be its fraction field.   Let $\m$ be the maximal ideal of $R$ and let $k=R/\m$ be the residue field.  Assume that $k$ does not have characteristic $2$.

Now fix an elliptic curve $E$ over $K$ as in \S\ref{SS:2-descent}.  With notation as in \S\ref{SS:2-descent}, we have an elliptic curve $E'$ over $K$ and a homomorphism 
\[
\delta_{E'}\colon E'(K)\to K^\times/(K^\times)^2.
\]
We can identify $R^\times/(R^\times)^2$ with a subgroup of $K^\times/(K^\times)^2$.  Using the discrete valuation ring $R$, we can apply Tate's algorithm \cite[Algorithm 9.4]{SilvermanII} to the curves $E$ and $E'$ to obtain Kodaira symbols.   

\begin{lemma} \label{L:HP inclusion}
Suppose that $E$ and $E'$ have Kodaira symbol $\operatorname{I}_{2n}$ and $\operatorname{I}_n$, respectively, for some $n\geq 0$. 
\begin{romanenum}
\item \label{L:HP inclusion i}
If $E$ has split multiplicative reduction or $n$ is odd, then $\Im(\delta_{E'})=1$.
\item \label{L:HP inclusion ii}
If $E$ and $E'$ have good reduction, i.e., $n=0$, then $\Im(\delta_{E'}) \subseteq R^\times/(R^\times)^2$.
\end{romanenum}
\end{lemma}
\begin{proof}
Let $v\colon K \to \ZZ \cup \{+\infty\}$ be the discrete valuation normalized so that $v(K^\times)=\ZZ$.  Choose an element $\pi\in R$ for which $v(\pi)=1$.  Since $R$ is a complete discrete valuation ring and $k$ has characteristic not equal to $2$, we find that an element of $R^\times$ is a square if and only if its image in $k$ is a square.

We may choose a model (\ref{E:basic model}) for $E/K$ with coefficients in $R$ so that it is a minimal Weierstrass model over $R$.   The curve $E'/K$ is given by the model $y^2=x^3-2ax+(a^2-4b)x$.  By our assumptions on the Kodaira symbols, we have $n=v(b)$ and $a^2-4b \in R ^\times$.   In particular, $(a^2-4b)\cdot (K^\times)^2$ lies in $R^\times/(R^\times)^2\subseteq K^\times/(K^\times)^2$.   

Suppose that $n\geq 1$.  We have $v(a)=0$ since otherwise $v(a^2-4b)>0$.  Observe that $a^2-4b$ is a square in $K^\times$ since $a^2-4b$ lies in $R^\times$ and $a^2-4b\equiv a^2 \pmod{\m}$.  Therefore, $(a^2-4b)\cdot (K^\times)^2=1$ when $n\geq 1$.

Now take any $\alpha \in \Im(\delta_{E'})-\{1,(a^2-4b) \cdot (K^\times)^2\}$.  We have $\alpha=d\cdot (K^\times)^2$ for some $d\in K^\times$ that satisfies $v(d)\in \{0,1\}$.  Therefore, $\alpha=\delta_{E'}((x,y))$ for some $(x,y)\in E'(K)$ where $x=dz^2$ with $z\in K^\times$. We have $y^2=x^3-2ax^2+(a^2-4b)x$ and hence
\begin{align}\label{E:sieve}
w^2 = d z^4 -2a z^2 + (a^2-4b)/d
\end{align}
for some $w\in K$.   Since $d$ is not a square in $K$ we deduce that $d$ is not square modulo $\m$ when $v(d)=0$.

Suppose that $v(z)<0$.    We have $v(dz^4)=4v(z) + v(d) \leq 4 v(z) + 1$.  Therefore, $v(dz^4) < 2v(z)\leq v(-2az^2)$ and $v(dz^4) < - 1 \leq  v((a^2-4b)/d)$.   From (\ref{E:sieve}) and these valuations, we deduce that $2v(w)=v(w^2)$ is equal to $v(dz^4)=4v(z) + v(d)$.  This gives a contradiction when $v(d)$ is odd, so we have $v(d)=0$ and $v(w)=2v(z)$.   With $e:=-v(z)\geq 1$, we have $z=\pi^{-e} z_0$ and $w=\pi^{-2e} w_0$ with $w_0,z_0\in R^\times$.  Multiplying both sides of (\ref{E:sieve}) by $\pi^{4e}$ and reducing modulo $\m$, we deduce that $w_0^2\equiv d z_0^4 \pmod{\m}$.    Thus $d$ is a square modulo $\m$ which is a contradiction.  

Suppose that $v(z)\geq 0$ and $v(d)=1$.  From (\ref{E:sieve}) we deduce that $2v(w)=v(w^2)$ equals $v((a^2-4b)/d)=-v(d)=-1$ which gives a contradiction.

Suppose that $v(z)>0$ and $v(d)=0$. From (\ref{E:sieve}) we deduce that $2v(w)=v(w^2)$ equals $v((a^2-4b)/d)=-v(d)=0$.  So $w\in R^\times$ and $w^2\equiv (a^2-4b)/d \pmod{\m}$.  Since $(a^2-4b)/d$ lies in $R^\times$ and is a square modulo $\m$, we deduce that $(a^2-4b)/d$ is a square in $K$.   However, this is impossible since $\alpha=d\cdot (K^\times)^2$ was chosen not to be equal to $(a^2-4b)\cdot (K^\times)^2$.

From the above cases, we find that $v(z)=0$ and $v(d)=0$.    Since $d\in R^\times$, we have $\alpha \in R^\times/(R^\times)^2$.  This completes the proof of (\ref{L:HP inclusion ii}).  We may now assume that $n\geq 1$ and hence $v(a)=0$.    Multiplying (\ref{E:sieve}) by $d$ and completing the square gives
\begin{align} \label{E:sieve 2}
dw^2=(dz^2 -a)^2 -4b.
\end{align}
We have $w\in R$ and reducing gives $dw^2 \equiv (dz^2 -a)^2 \pmod{\m}$.  Since $d$ is not a square modulo $\m$, we must have $dz^2-a\equiv 0 \pmod{\m}$.   The congruence $dz^2\equiv a \pmod{\m}$ and $v(d)=v(a)=0$ implies that $\alpha=d \cdot (K^\times)^2= a \cdot (K^\times)^2$.   

Suppose that $E$ has split reduction.   Since $v(a)=0$, this implies that $a$ is a square modulo $\m$ and hence $a$ is a square in $K$.   We have $\alpha= a \cdot (K^\times)^2=1$ which contradicts our choice of $\alpha$.

Suppose that $n$ is odd.  Since $v(b)=n$ is odd and $v(d)=0$, (\ref{E:sieve 2}) implies that the even integers $v(dw^2)$ and $v((dz^2 -a)^2)$ agree.  So there is an integer $e$ such that $w=\pi^e w_0$ and $dz^2-a=\pi^e u_0$ with $u_0,w_0\in R^\times$.  Dividing (\ref{E:sieve 2}) by $\pi^{2e}$ gives $dw_0^2=u_0^2-4b/\pi^{2e}$.  Since $v(b)$ is odd, this implies that $dw_0^2\equiv u_0^2 \pmod{\m}$ which contradicts that $d$ is not a square modulo $\m$.

In the setting of (\ref{L:HP inclusion i}), we have now proved that $\delta_{E'}$ has trivial image.
\end{proof}

\section{General theorem}
\label{S:general}

Let $K$ be a number field.  Let $E$ be an elliptic curve defined over the function field $K(T)$ with a fixed point $P_0\in E(K(T))$ of order $2$.   There is a degree $2$ isogeny $\phi\colon E \to E'$ whose kernel is generated by $P_0$.   Using the isogeny $\phi$ and its dual, we obtain group homomorphisms
\[
\delta_{E}\colon E(K(T))\to K(T)^\times/(K(T)^\times)^2\quad \text{ and }\quad \delta_{E'}\colon E'(K(T))\to K(T)^\times/(K(T)^\times)^2
\]
as in \S\ref{SS:2-descent}.

We will identify $K(T)$ with the function field of $\PP^1_K=\Spec K[T] \cup \{\infty\}$.  For a closed point $P$ of $\PP^1_K$, let $v_P$ be the corresponding discrete valuation of $K(T)$.  Let $K_P$ be the completion of $K$ with respect to $v_P$ and let $\OO_P$ be its valuation ring.   After base extending our curves $E$ and $E'$ to $K_P$ we can apply Tate's algorithm to obtain Kodaira symbols and Tamagawa numbers $c_P(E)$ and $c_P(E')$.

Let $\calB$ be the set of closed points of $\PP^1_K$ over which $E$ has bad reduction.   Let $\calA$ be the set of points in $\calB$ over which $E$ has additive reduction.   Let $\calM$ be the set of $P\in \calB-\calA$ for which the the Kodaira symbol of $E$ and $E'$ at $P$ is $\operatorname{I}_{2n}$ and $\operatorname{I}_n$, respectively, for some $n\geq 1$.
Let $\calM'$ be the set of $P\in \calB-\calA$ for which the the Kodaira symbol of $E$ and $E'$ at $P$ is $\operatorname{I}_{n}$ and $\operatorname{I}_{2n}$, respectively, for some $n\geq 1$.
We have a disjoint union
\[
\calB = \calA \cup \calM \cup \calM'.
\]

We also impose the following additional conditions on $E$ and $E'$: 
\begin{alphenum} 
\item \label{I:new a}
Every point in $\calB$ has degree $1$; equivalently, $\calB$ can be viewed as a subset of $\PP^1(K)$.
\item \label{I:new b}
The sets $\calM$ and $\calM'$ are nonempty.
\item \label{I:new c}
For every point $P\in \calM$, the curve $E'$ has split multiplicative reduction at $P$ or its Kodaira symbol at $P$ is $\operatorname{I}_n$ for some odd $n$.
\item \label{I:new d}
For every point $P\in \calM'$, the curve $E$ has split multiplicative reduction at $P$ or its Kodaira symbol at $P$ is $\operatorname{I}_n$ for some odd $n$.
\item \label{I:new e}
If the Kodaira symbol of $E$ at a point $P\in \calA$ is $\operatorname{I}_n^*$ for some $n\geq 0$, then $c_P(E)=c_P(E')=4$.
\item \label{I:new f}
We have $\dim_{\FF_2} \delta_E(E(K(T))) \geq |\calA|+|\calM|-1$.
\item \label{I:new g}
We have $\dim_{\FF_2} \delta_{E'}(E'(K(T))) \geq |\calA|+|\calM'|-1$.
\end{alphenum}

Define the integer
\[
r:=2|\calA|+|\calM|+|\calM'|-4.
\]
By the Lang--N\'eron theorem, we know that $E(K(T))$ is a finitely generated abelian group.  By Lemma~\ref{L:rank via quotients} with conditions (\ref{I:new f}) and (\ref{I:new g}), we find that $\rank\, E(K(T)) \geq r$.  However, the rank of $E(K(T))$ is at most $r$ by \cite[Corollary~2]{MR1211006} and (\ref{I:new a}).  Therefore,
\[
r=\rank\, E(K(T)).
\]
Given an explicit Weierstrass model of $E$, we can specialize it at all but finitely many $t\in K$ to obtain an elliptic curve $E_t$ over $K$.    Our main result is the following theorem which will be proved in \S\ref{S:proof of general}.

\begin{thm} \label{T:general}
With assumptions as above, there are infinitely many $t\in K$ for which $E_t(K)$ has rank $r$.   In particular, there are infinitely many elliptic curves over $K$, up to isomorphism over $\Kbar$, that have rank $r$.
\end{thm}

\section{Proof of Theorem~\ref{T:general}} \label{S:proof of general}

Set $L:=K(T)$.   After possibly replacing $T$ by another element that generates the field $L$ over $K$, we may assume without loss of generality that  
\[
\calB \subseteq \OO_K\subseteq K \cup \{\infty\} = \PP^1(K).
\]
We can choose a Weierstrass model 
\begin{align} \label{E:basic eqn}
y^2=x^3+ax^2+bx
\end{align}
of $E$ with $a,b\in \OO_K[T]$ so that $P_0=(0,0)$ is our fixed $2$-torsion point.  We may further assume that the Weierstrass model (\ref{E:basic eqn}) is minimal over the PID $K[T]$.     From \S\ref{SS:2-descent}, we may assume that $E'$ is given by the Weierstrass model $y^2=x^3-2a x^2 + (a^2-4b) x$.   The discriminant of our models for $E$ and $E'$ are $ \Delta=16(a^2-4b)b^2$ and $\Delta'=16(a^2-4b)^2b$, respectively.  

Using the minimality of this model for $E$, we can now explicitly describe the sets $\calB$, $\calA$, $\calM$ and $\calM'$ in terms of $a$ and $b$. From condition (\ref{I:new a}) and $\calB\subseteq \OO_K$,  we find that $\calB$ is the set of roots in $\Kbar$ of the polynomial $\Delta$.   The set $\calM$ is equal to the set of roots of $b$ that are not roots of $a^2-4b$.  The set $\calM'$ is equal to the set of roots of $a^2-4b$ that are not roots of $b$.   The set $\calA$ is equal to the set of common roots of $a^2-4b$ and $b$.   

Note that with these Weierstrass equations, the specializations $E_t$ and $E'_t$ are well-defined for all $t\in K-\calB$.

\subsection{Geometric Mordell--Weil}

We can choose points $e_0\in \calM$ and $e_0'\in \calM'$ by condition (\ref{I:new b}).  Let $\scrF$ be the set of squarefree polynomials $f\in K[T]$ that divide $b$, have even degree, and satisfy $f(e_0')=1$.    Let $\scrF'$ be the set of squarefree polynomials $f\in K[T]$ that divide $a^2-4b$, have even degree, and satisfy $f(e_0)=1$.  We can now describe the image of $\delta_E$ and $\delta_{E'}$. 

\begin{lemma} \label{L:geometric ranks}
\begin{romanenum}
\item  \label{L:geometric ranks i}
We have $\delta_{E}(E(K(T)))=\{ f\cdot (K(T)^\times)^2 : f\in \scrF\}$.
\item   \label{L:geometric ranks ii}
We have $\delta_{E'}(E'(K(T)))=\{ f\cdot (K(T)^\times)^2 : f\in \scrF'\}$.
\end{romanenum}
\end{lemma}
\begin{proof}
Take any $\alpha\in \delta_E(E(K(T))$.   We have $\alpha = f \cdot  (K(T)^\times)^2$ for some squarefree polynomial $f\in K[T]$.   Take any closed point $P$ of $\PP^1_K$ that is not in $\calA\cup \calM$.  If $P\notin \calM'$, then $E$ and $E'$ base extended to $K(T)_P$ both have good reduction.   If $P\in \calM'$, then condition (\ref{I:new d}) implies that $E$ base extended to $K(T)_P$ has split multiplicative reduction or has Kodaira symbol $\operatorname{I}_n$ for some odd $n$.   Lemma~\ref{L:HP inclusion} thus implies that $v_P(f)$ is even for all closed points $P\notin \calA\cup \calM$.     Since $K[T]$ is a PID and $f$ is squarefree, we deduce that $f$ divides $b$; recall the set of roots of $b$ is $\calA\cup \calM$. Since $\infty \in \PP^1(K)$ does not lie in $\calA\cup \calM\cup \calM'=\calB$, $v_\infty(f)$ being even implies that $f$ has even degree.   Therefore, $f=cf_1$ for a unique polynomial $f_1\in \scrF$ and constant $c\in K^\times$.   Lemma~\ref{L:HP inclusion}(\ref{L:HP inclusion i}) implies that $f\cdot (K(T)_{e_0'}^\times)^2=1$ and hence $f(e_0')=c$ is a square in $K$.  Therefore, $\alpha=f \cdot (K(T)^\times)^2 = f_1 \cdot (K(T)^\times)^2$ with $f_1\in \scrF$.  We have proved the inclusion 
\[
\delta_E(E(K(T))) \subseteq\{ f\cdot (K(T)^\times)^2 : f\in \scrF\}=:G.
\]   
It is actually an equality since $|G|=2^{|\calA|+|\calM|-1}$ and $|\delta_E(E(K(T))) |\geq 2^{|\calA|+|\calM|-1}$ by condition (\ref{I:new f}).  This proves (\ref{L:geometric ranks i}).   Part (\ref{L:geometric ranks ii}) is proved in the same way but making use of condition (\ref{I:new g}).
\end{proof}

We now consider specializations.

\begin{lemma} \label{L:geometric inclusion}
There is a finite set $D\subseteq K$ with $\calB \subseteq D$ such that the inclusions
\[
\delta_{E_t}(E_t(K))\supseteq \{ f(t)\cdot (K^\times)^2 : f\in \scrF\}\quad \text{ and }\quad  \delta_{E'_t}(E'_t(K))\supseteq \{ f(t)\cdot (K^\times)^2 : f\in \scrF'\}
\]
hold for all $t\in K-D$.
\end{lemma}
\begin{proof}
Let $D$ be a finite subset of $K$ containing $\calB$ that we can extend a finitely number of times.  Take any $t\in K-D$.

Take any $f\in \scrF$.  We have $f(t)\neq 0$ since $t\notin \calB$ and $f$ divides $\Delta$.  By Lemma~\ref{L:geometric ranks}(\ref{L:geometric ranks i}), there is a point $P\in E(K(T))$ such that $\delta_E(P)=f\cdot (K(T)^\times)^2$.   If $P=0$, then $f$ is a square in $K(T)$  and hence $f(t)\cdot (K^\times)^2= (K^\times)^2$ is an element of $\delta_{E_t}(E_t(K))$.  Suppose that $P=(0,0)$ and hence $f\cdot (K(T)^\times)^2= \delta_E(P) = b\cdot (K(T)^\times)^2$.  Since the roots of $f$ and $b$ both lie in $\calB$, we have $f(t)\cdot (K^\times)^2 = b(t)\cdot (K^\times)^2 = \delta_{E_t}((0,0))$ which proves that $f(t)\cdot (K^\times)^2\in \delta_{E_t}(E_t(K))$.  Now suppose that $P=(x,y)$ with $x,y\in K(T)$ and $x\neq 0$.  We have $x=f z^2$ for some nonzero $z\in K(T)$.   After increasing the finite set $D$, we may assume that $z$ and $y$ have no poles at $t$ and that $z(t)\neq 0$.  Therefore, $P_t:=(x(t),y(t))$ is a point in $E_t(K)$ and $\delta_{E_t}(P_t)= x(t) \cdot (K^\times)^2 = f(t) \cdot (K^\times)^2$.

Since $\scrF$ is finite, after increasing the finite set $D$ appropriately we will have $f(t)\cdot (K^\times)^2 \in \delta_{E_t}(E_t(K))$ for all $t\in K-D$ and all $f\in \scrF$.  This proves the first inclusion of the lemma.   The second inclusion is proved in the exact same way by making use of Lemma~\ref{L:geometric ranks}(\ref{L:geometric ranks ii}).  
\end{proof}

\subsection{Choice of primes} \label{SS:choice of primes}

In this section, we construct a set $S\cup \{\q_1,\q_2\}$ of nonzero prime ideals of $\OO_K$ that satisfy various properties.  These primes will be used in \S\ref{SS:choice of elliptic curve} to choose a specialization $E_t$ for which we can compute the appropriate Selmer groups and prove that $E_t$ has rank $r$.

Let $S_0$ be a finite set of nonzero prime ideals of $\OO_K$ that contains all those that divide $6 \prod_{\alpha,\beta\in \calB, \alpha\neq \beta} (\alpha-\beta) \in \OO_K$ or divide the leading coefficient of $\Delta$.  By adding prime ideals to the finite set $S_0$, we may further assume that $\OO_{K,S_0}$ is a PID.

\begin{lemma} \label{L:sieve input}
For all $P\in \calB$ and all but finitely many nonzero prime ideals $\p\notin S_0$ of $\OO_K$, if $t\in K$ satisfies $v_\p(t-P)=1$, then 
\[
\tfrac{1}{2} |\Im(\delta_{E'_t,\p})| = 
\begin{cases}
2 & \text{ if $P\in \calM$,}\\
\tfrac{1}{2} & \text{ if $P\in \calM'$,}\\
1 & \text{ if $P\in \calA$.}
\end{cases}
\]
\end{lemma}
\begin{proof}
Take any $P\in \calB$.  Take any nonzero prime ideal $\p\notin S_0$ of $\OO_K$ and $t\in K$ so that $v_\p(t-P)=1$.  For each $e\in \calB-\{P\}$, we have $v_\p(t-e)=0$ since otherwise $v_\p(P-e)>0$ which contradicts $\p\notin S_0$.  In particular, $t\notin \calB$.  By enlarging $S_0$ appropriately, we will have $v_\p(\Delta(t))=v_P(\Delta)$ and $v_\p(\Delta'(t))=v_P(\Delta')$.  Since $\p\nmid 2$, we have 
\[
\tfrac{1}{2} |\operatorname{Im}(\delta_{E_t',\p})| = c_\p(E'_t)/c_\p(E_t),
\] 
by Lemma~\ref{L:Selmer ratio}(\ref{L:Selmer ratio ii}).

We can apply Tate's algorithm \cite[Algorithm 9.4]{SilvermanII} to $E$ over $K(T)_P$; it produces a minimal Weierstrass model over $\OO_P$, using the uniformizer $\pi:=T-P$, that satisfies various properties before the algorithm completes.   Substituting $t$ for $T$, and by increasing the finite set $S_0$ appropriately in a way that does not depend on $t$, the assumption $v_\p(t-P)=1$ ensures that the model obtained is a minimal model for $E_t$ at $\p$ and that $E$ and $E_t$ have the same Kodaira symbol $\kappa$ at $P$ and $\p$, respectively.    Similarly by increasing $S_0$, we may also assume that $E'$ and $E'_t$ have the same Kodaira symbol $\kappa'$ at $P$ and $\p$, respectively.   

Consider the case where $\kappa\in \{\operatorname{II}, \operatorname{III}, \operatorname{IV}, \operatorname{II}^*, \operatorname{III}^*, \operatorname{IV}\}$ and hence $E_t$ has potentially good reduction at $\p$.  By \cite[Table~1]{MR3324930},  we have $c_\p(E'_t)/c_\p(E_t)=1$. 

Consider the case where $\kappa=\operatorname{I}_n^*$ for some $n\geq 0$.   Since $E$ and $E'$ are $2$-isogenous to each other, we find that $\kappa'=\operatorname{I}_m^*$ for some $m\geq 0$.  By condition (\ref{I:new e}), we have $c_P(E)=c_P(E')=4$.  In Tate's algorithm, the conditions $c_P(E)=4$ and $c_P(E')=4$ are equivalent to certain polynomials with coefficients in $\OO_P$ having distinct roots and splitting completely when reduced modulo the maximal ideal $\m_P$ of $\OO_P$.   Substituting $t$ for $T$ we deduce that $c_\p(E_t)=4$ and $c_\p(E'_t)=4$ after possibly enlarging $S_0$.  Therefore, $c_\p(E'_t)/c_\p(E_t)=1$. 

We have now verified that $\tfrac{1}{2} |\operatorname{Im}(\delta_{E_t',\p})| =1$ when $P\in \calA$.  It remains to consider the cases where $P\in \calM\cup\calM'$.

Consider the case $P\in \calM'$.  We have $\kappa=\operatorname{I}_n$ and $\kappa'=\operatorname{I}_{2n}$ for some $n\geq 1$; the integer $n$ is the $\p$-adic valuation of the discriminant of a Weierstrass  model of $E_t$ that is minimal at $\p$.  By condition (\ref{I:new d}), $E$ has split reduction at $P$ or $n$ is odd.  After possibly increasing the set $S_0$, this implies that $E_t$ has split reduction at $\p$ or $n$ is odd.    By \cite[Table~1]{MR3324930},  we have $c_\p(E'_t)/c_\p(E_t)=1/2$. 

Finally consider the case $P\in \calM$.  We have $\kappa=\operatorname{I}_{2n}$ and $\kappa'=\operatorname{I}_{n}$ for some $n\geq 1$; the integer $n$ is the $\p$-adic valuation of the discriminant of a Weierstrass model of $E'_t$ that is minimal at $\p$.  By condition (\ref{I:new c}), $E'$ has split reduction at $P$ or $n$ is odd.  After possibly increasing the set $S_0$, this implies that $E'_t$ has split reduction at $\p$ or $n$ is odd.    By \cite[Table~1]{MR3324930},  we have $c_\p(E'_t)/c_\p(E_t)=2$. 
\end{proof}

\begin{lemma} \label{L:good reduction at infinity}
For all but finitely many nonzero prime ideals $\p\notin S_0$ of $\OO_K$, if $t\in K-\calB$ satisfies $v_\p(t)<0$, then $E_t$ and $E_t'$ both have good reduction at $\p$.
\end{lemma}
\begin{proof}
Set $U:=T^{-1}$.  We can choose a Weierstrass model $y^2=x^3+cx+d$ for $E$, with $c,d\in \OO_K[U]$, that is minimal over the PID $K[U]$.  Let $\tilde{\Delta}\in \OO_K[U]$ be the discriminant of this model.   Since $E$ has good reduction at $\infty\in \PP^1(K)$, we find that $U\nmid \tilde{\Delta}$.  Take any $\p\notin S_0$ that does not divide $\tilde{\Delta}(0)\in \OO_K$.  Take any $t\in K-\calB$ with $v_\p(t)<0$.  Substituting $U$ by $t^{-1}\in \p\OO_p$, we obtain a Weierstrass model for $E_t$ with coefficients in $\OO_\p$ whose discriminant $\tilde{\Delta}(t^{-1})$ lies in $\OO_\p^\times$.  Therefore, $E_t$ has good reduction at $\p$.   The curve $E'_t$ also has good reduction at $\p$ since it is isogenous to $E_t$.
\end{proof}

By increasing the finite set $S_0$ we shall now assume that the conclusions of Lemmas~\ref{L:sieve input} and \ref{L:good reduction at infinity} hold for all nonzero prime ideals $\p\notin S_0$ of $\OO_K$.  Let $S_\infty$ be the set of infinite places of $K$.

\begin{lemma}  \label{L:balance}
There is a finite set $S\supseteq S_0$ of nonzero prime ideals of $\OO_K$ such that for each  $v\in S\cup S_{\infty}$ we have a nonempty open subset $U_v$ of $K_v$ and a subgroup $\Phi_{v} \subseteq K_v^\times/(K_v^\times)^2$ for which the following hold:
\begin{alphenum}
\item \label{L:balance a}
For each $v \in S\cup S_\infty$ and $t\in K \cap U_v$, we have $t\notin D$ and  $\Im(\delta_{E'_t,v}) = \Phi_v$.
\item \label{L:balance b}
The natural map
\[
\OO_{K,S}^\times/(\OO_{K,S}^\times)^2 \to \prod_{v\in S\cup S_\infty} (K_v^\times/(K_v^\times)^2)/\Phi_v
\]
 is a group isomorphism.
\item
\label{L:balance c}
For each $v\in S$ dividing $2$, $U_v\cap \OO_v = \emptyset$.
\end{alphenum}
\end{lemma}
\begin{proof}
Take any place $v\in S_0 \cup S_\infty$ and choose an element $u_v \in K_v^\times$ with $u_v\notin D$. When $v|2$, we assume $u_v$ has been chosen so that $u_v\notin \OO_v$.   With a fixed real number $\epsilon_v>0$, define the open subset 
\[
U_v:=\{u\in K_v:  |u-u_v|_v<\epsilon_v \}
\] 
of $K_v$, where $|\cdot |_v$ is an absolute value on $K_v$ corresponding to $v$.   By taking $\epsilon_v>0$ sufficiently small, we may assume that $u\notin D$ for all $u\in U_v$.  We can also assume that (\ref{L:balance c}) holds when $v|2$.  For each $u\in U_v$, we have $\Delta(u)\neq 0$ since $u\notin D$ and hence specialization of $E'$ at $u$ gives an elliptic curve $E'_u$ over $K_v$.     For $u\in U_v$, we have a homomorphism
\begin{align} \label{E:Kv connecting}
\delta_{E'_u}\colon E_{u}'(K_v) \to K_v^\times/(K_v^\times)^2
\end{align}
as in \S\ref{SS:2-descent}; we have $\delta_{E'_u}((x,y))=x\cdot (K_v^\times)^2$ for $(x,y) \in E'_u(K_v)- \{0,(0,0)\}$ and $\delta_{E'_u}((0,0))= (a^2-4b) \cdot (K_v^\times)^2$.   The key topological observation is that the image of (\ref{E:Kv connecting}) does not depend on the choice $u\in U_v$ if we take $\epsilon_v>0$ sufficiently small.   Thus by taking $\epsilon_v>0$ sufficiently small, we may assume that (\ref{E:Kv connecting}) has a common image $\Phi_{v}$ for all $u\in U_v$.   In particular, for all $t\in K\cap U_v$ we have $\Im(\delta_{E'_t,v})=\Phi_v$.

By weak approximation for $K$, there is a finite set $S_1$ of nonzero prime ideals of $\OO_K$ that is disjoint from $S_0$ such that the natural homomorphism 
\[
\psi\colon \OO_{K,S_0\cup S_1}^\times/(\OO_{K,S_0\cup S_1}^\times)^2 \to \prod_{v\in S_0\cup S_\infty} (K_v^\times/(K_v^\times)^2)/\Phi_v
\] 
is surjective.   Choose units $u_1,\ldots, u_m\in \OO_{K,S_0\cup S_1}^\times$ whose  image in $\OO_{K,S_0\cup S_1}^\times/(\OO_{K,S_0\cup S_1}^\times)^2$ gives a basis of the $\FF_2$-vector space $\ker \psi$.   The field $L:=K(\sqrt{u_1},\ldots, \sqrt{u_m})$ is an abelian extension of $K$ with $\Gal(L/K)\cong (\ZZ/2\ZZ)^m$.  By the Chebotarev density theorem, there are nonzero prime ideals $\p_1,\ldots, \p_m\notin S_0\cup S_1$ of $\OO_K$ such that $u_i$ is a square in $\OO_{\p_j}^\times$ if and only if $i\neq j$.    Define $S_2:=\{\p_1,\ldots, \p_m\}$.   Observe that if $u\cdot (\OO_{K,S_0\cup S_1}^\times)^2$ is in the kernel of $\psi$ and $u$ is a square in $\OO_v^\times$ for all $v\in S_2$, then $u \cdot (\OO_{K,S_0\cup S_1}^\times)^2=1$.

With $S:=S_0 \cup S_1 \cup S_2$, we define the natural group homomorphism
\[
\varphi\colon \OO_{K,S}^\times/(\OO_{K,S}^\times)^2 \to \prod_{v\in S\cup S_\infty} (K_v^\times/(K_v^\times)^2)/\Phi_v,
\]
where $\Phi_v:=K_v^\times/(K_v^\times)^2$ if $v\in S_1$ and $\Phi_v:=1$ if $v\in S_2$. 

We claim that $\varphi$ is injective.  Take any $\alpha\in \ker \varphi$ and choose a unit $u\in \OO_{K,S}^\times$ that represents $\alpha$.   Since $\varphi(\alpha)=1$, we find that $u$ is a square in $K_v^\times$ for all $v\in S_2$.   The ring $\OO_{K,S_0}$, and hence also $\OO_{K,S_0\cup S_1}$, is a PID by our choice of $S_0$.   Thus after multiplying $u$ by a suitable square in $K$, we may assume that our chosen $u$ is an element of $\OO_{K,S_0\cup S_1}^\times$.   Since $\varphi(\alpha)=1$, we have $u\cdot (K_v^\times)^2\in \Phi_v$ for all $v\in S_0\cup S_\infty$.    Therefore, $u\cdot (\OO_{K,S_0\cup S_1}^\times)^2$ is in the kernel of $\psi$.    Since $\varphi(\alpha)=1$, we have $u\cdot (K_v^\times)^2\in \Phi_v=1$ for all $v\in S_2$.   Since $u\cdot (\OO_{K,S_0\cup S_1}^\times)^2$ is in the kernel of $\psi$ and $u$ is a square in $\OO_v^\times$ for all $v\in S_2$, we deduce that $u$ is a square in $\OO_{K,S_0\cup S_1}^\times$ and hence $\alpha=1$.  This completes the proof of the claim.

We claim that $\varphi$ is an isomorphism.  We have
\[
\prod_{v\in S\cup S_\infty} |(K_v^\times/(K_v^\times)^2)/\Phi_v| = 4^m \prod_{v\in S_0 \cup S_\infty} |(K_v^\times/(K_v^\times)^2)/\Phi_v| = 2^m \cdot |\OO_{K,S_0\cup S_1}^\times/(\OO_{K,S_0\cup S_1}^\times)^2|,
\]
where the first equality uses that $|K_v^\times/(K_v^\times)^2/\Phi_v|$ is $1$ for $v\in S_1$ and $4$ for $v\in S_2$, and the last equality uses that $\psi$ is surjective with kernel of cardinality $2^m$.   By Dirichlet's unit theorem, the groups $\OO_{K,S_0\cup S_1}^\times$ and $\OO_{K,S}^\times$ are finitely generated abelian groups of rank $|S_\infty|+|S_0|+|S_1|-1$ and $|S_\infty|+|S|-1$, respectively.  Since the torsion subgroups of these unit groups are cyclic of even order, we have 
\[
|\OO_{K,S}^\times/(\OO_{K,S}^\times)^2|=2^{|S_\infty|+|S|} = 2^{|S_2|} \cdot 2^{|S_\infty|+|S_0|+|S_1|} = 2^m \cdot |\OO_{K,S_0\cup S_1}^\times/(\OO_{K,S_0\cup S_1}^\times)^2|.
\]
We have now shown that the domain and codomain of $\varphi$ have the same finite cardinality.  The claim follows since we have already proved that $\varphi$ is injective.

This proves (\ref{L:balance b}) and we have also verified (\ref{L:balance a}) for $v\in S_0\cup S_\infty$.

Take any $v\in S_1$ and set $\p:=v$.  There is an element  $e\in \calM$ by (\ref{I:new b}).  Since $v\notin S_0$, there is a $t_v \in K$ with $v_\p(t_v-e)=1$ and $v_\p(t_v-c)=0$ for $c\in \calB-\{e\}$.  By Lemma~\ref{L:sieve input}, we have $|\Im(\delta_{E'_{t_v,v}})|=4$ and hence $\Im(\delta_{E'_{t_v,v}})=K_v^\times/(K_v^\times)^2=\Phi_v$.   Arguing as above, there is an open neighborhood $U_v$ of $t_v$ in $K_v$ such that $\delta_{E_u'}\colon E_u'(K_v)\to K_v^\times/(K_v^\times)^2$ has image $\Phi_v$ for all $u\in U_v$.   This gives (\ref{L:balance a}) for $v\in S_1$.

Take any $v\in S_1$ and set $\p:=v$.  There is an element  $e\in \calM'$ by (\ref{I:new b}).  Since $v\notin S_0$, there is a $t_v \in K$ with $v_\p(t_v-e)=1$ and $v_\p(t_v-c)=0$ for $c\in \calB-\{e\}$.  By Lemma~\ref{L:sieve input}, we have $|\Im(\delta_{E'_{t_v,v}})|=1$ and hence $\Im(\delta_{E'_{t_v,v}})=1=\Phi_v$.   Arguing as above, there is an open neighborhood $U_v$ of $t_v$ in $K_v$ such that $\delta_{E_u'}\colon E_u'(K_v)\to K_v^\times/(K_v^\times)^2$ has image $\Phi_v$ for all $u\in U_v$.   This gives (\ref{L:balance a}) for $v\in S_2$.
\end{proof}

For the rest of the proof, we fix $S$, $\{U_v\}_{v\in S\cup S_\infty}$ and $\{\Phi_v\}_{v\in S\cup S_\infty}$ as in Lemma~\ref{L:balance}.  The following lemma will be key when we later compute the ratio of cardinalities of Selmer groups.
 
\begin{lemma} \label{L:balanced product}
We have $\prod_{v\in S \cup S_\infty} \tfrac{1}{2} |\Phi_v|=1$.
\end{lemma}
\begin{proof}
Let $r$ be the number of real embeddings of $K$ and let $s$ be the number of pairs of conjugate complex embeddings of $K$.   By Dirichlet's unit theorem, the abelian group $\OO_{K,S}^\times$ is finitely generated and has rank $r+s+|S|-1$.  Since the torsion subgroup of $\OO_{K,S}^\times$ is cyclic of even order, we deduce that 
\begin{align} \label{E:unit theorem rank}
|\OO_{K,S}^\times/(\OO_{K,S}^\times)^2|=2^{(r+s+|S|-1)+1}=2^{r+s+|S|}.
\end{align}

Take any place $v|2$ of $K$.   We have $[K_v:\QQ_2]=e_vf_v$, where $e_v$ and $f_v$ are the ramification index and inertia degree, respectively, of the extension $K_v/\QQ_2$.  We have $\OO_v^\times\cong C_v \times \ZZ_2^{e_vf_v}$, where $C_v$ is a finite cyclic group, cf.~\cite[II 5.7]{Neukirch}.  The group $C_v$ has even cardinality since it contains $-1$ and hence $\OO_v^\times/(\OO_v^\times)^2\cong (\ZZ/2\ZZ)^{1+e_vf_v}$.   Therefore, $\tfrac{1}{2} |K_v^\times/(K_v^\times)^2| = 2^{1+e_v f_v}$.  Since $\sum_{v|2} e_v f_v = [K:\QQ]=r+2s$, we deduce that $\prod_{v|2} \tfrac{1}{2} |K_v^\times/(K_v^\times)^2| = 2^{r+2s} \prod_{v|2} 2$.    For any place $v\in S$ with $v\nmid 2$, we have $\tfrac{1}{2}|K_v^\times/(K_v^\times)^2\big|=2$.   We have $\prod_{v\in S_\infty} \tfrac{1}{2} |K_v^\times/(K_v^\times)^2|=1/2^s$.   Using these computations and that $S$ contains all the places of $K$ dividing $2$, we find that
\begin{align} \label{E:prod Phiv is 1}
\prod_{v\in S\cup S_\infty} \tfrac{1}{2} | K_v^\times/(K_v^\times)^2| = 2^{-s}\cdot 2^{r+2s} \prod_{v|2} 2 \cdot \prod_{v\in S,\, v\nmid 2} 2  = 2^{r+s+|S|}.
\end{align}
Using the isomorphism from Lemma~\ref{L:balance}(\ref{L:balance b}), we have
\[
\prod_{v\in S \cup S_\infty} \tfrac{1}{2} |\Phi_v| = \prod_{v\in S\cup S_\infty} \tfrac{1}{2} | K_v^\times/(K_v^\times)^2|  \cdot |\OO_{K,S}^\times/(\OO_{K,S}^\times)^2|^{-1} = 1,
\]
where the last equality uses (\ref{E:unit theorem rank}) and (\ref{E:prod Phiv is 1}).
\end{proof}

\begin{lemma} \label{L:pi existence}
There is a nonzero $\pi \in \OO_K$ such that $\pi$  is a square in $K_v^\times$ for all $v\in S \cup S_\infty$ and $\pi\OO_K$ is a prime ideal of $\OO_K$ that does not lie in $S$.
\end{lemma}
\begin{proof}
Fix an integer $n_\p\geq 1$ for each $\p\in S$.  Let $I^S$ be the group of fractional ideals of $K$ generated by nonzero prime ideals $\p\notin S$ of $\OO_K$.   Let $H$ be the subgroup of $I^S$ consisting of $\alpha \OO_K$ where $\alpha\in K^\times$ is positive in $K_v$ for all real place $v$ in $S_\infty$ and $v_\p(\alpha-1) \geq n_\p$ for all $\p\in S$.   Define $C:=I^S/H$; it is the \emph{ray class group} of $K$ with modulus $\mathfrak{m}:=\prod_{\p\in S} \p^{n_\p} \cdot \prod_{v\in S_\infty \text{ real}} v$.   Let $K_{\mathfrak{m}}$ be the corresponding \emph{ray class field}; it is a finite abelian extension of $K$ and a prime ideal $\p \notin S$ splits completely in $K_{\mathfrak{m}}$ if and only if $\p\in H$.  By the Chebotarev density theorem, there is a nonzero prime ideal $\q\notin S$ of $\OO_K$ that splits completely in $K_{\mathfrak{m}}$.   We have $\q\in H$ and hence $\q=\pi \OO_K$ for some $\pi\in \OO_K$ that is positive in $K_v$ for all real place $v$ in $S_\infty$ and $v_\p(\pi-1) \geq n_\p$ for all $\p\in S$.  For each $v\in S_\infty$, $\pi$ is a square in $K_v$.    For each $\p \in S$, we have $\pi \in 1 +\p^{n_\p} \OO_\p$ and hence $\pi$ is a square in $K_\p$ assuming that $n_\p$ is large enough.   So by choosing the $n_\p \geq 1$ large enough, we find that $\pi$ is a square in $K_v^\times$ for all $v\in S\cup S_\infty$.
\end{proof}

For the rest of the proof, we fix $\pi\in \OO_K$ as in Lemma~\ref{L:pi existence} and define $\q_1:=\pi \OO_K$.

\begin{lemma} \label{L:q existence}
There is a nonzero prime ideal $\q_2\notin S\cup\{\q_1\}$ of $\OO_K$ such that every element in $\OO_{K,S\cup\{\q_1\}}^\times$ is a square in $K_{\q_2}$.
\end{lemma}
\begin{proof}
The group $\OO_{K,S\cup\{\q_1\}}^\times$ is finitely generated by Dirichlet's unit theorem; fix a finite set of generators $B$.  Define $L:=K(\{\sqrt{\alpha}: \alpha \in B\})$; it is a finite Galois extension of $K$.   By the Chebotarev density theorem, there is a nonzero prime ideal $\q_2\notin S\cup\{\q_1\}$ of $\OO_K$ that splits completely in $L$.    The lemma follows since every element of $B$ is a square in $K_{\q_2}$.
\end{proof}

For the rest of the proof, we fix a prime ideal $\q_2$ as in Lemma~\ref{L:q existence}.

\subsection{A choice of elliptic curve} \label{SS:choice of elliptic curve}
Assume we have a set of primes $S\cup\{\q_1,\q_2\}$ with assumptions as in \S\ref{SS:choice of primes}.

In this section, we will consider the elliptic curve $E_t$ over $K$ with $t:=m/n$, where $m$ and $n$ are chosen as in the following proposition.  It is here that we make use of a result of Kai that is a number field analogue of the Green–Tao–Ziegler theorems on simultaneous prime values of degree $1$ polynomials.

\begin{prop} \label{P:additive}
There are nonzero $m,n \in \OO_{K,S}$ such that the following hold:
\begin{alphenum}
\item \label{P:additive a}
the elements $m-en$ with $e\in \calB$ generate distinct nonzero prime ideals of $\OO_{K,S\cup \{\q_1,\q_2\}}$,
\item \label{P:additive b}
$m/n$ lies in $U_v$ for all $v\in S\cup S_\infty$,
\item \label{P:additive c}
$v_{\q_1}(m-en)=1$ for some $e\in \calM$, $v_{\q_1}(n)=0$, and $n$ is not a square modulo $\q_1$,
\item \label{P:additive d}
$v_{\q_2}(m-en)=1$ for some $e\in \calM'$, $v_{\q_2}(n)=0$, and $n$ is not a square modulo $\q_2$,
\item \label{P:additive e}
$v_\p(m-en)=0$ for all prime ideals $\p$ of $\OO_K$ dividing $2$.
\end{alphenum}
\end{prop}
\begin{proof}
For each $\p\in S$ with $\p\nmid 2$, let $\scrU_\p$ be the set of pairs $(\alpha,\beta)\in K_\p^2$ with $\beta\neq 0$ and $\alpha/\beta\in U_\p$.  For each $\p\in S$ with $\p| 2$, let $\scrU_\p$ be the set of pairs $(\alpha,\beta)\in K_\p$ with $\beta\neq 0$, $\alpha/\beta\in U_\p$, $v_\p(\alpha)=0$ and $v_\p(\beta) > 0$.   For each $\p\in S$, $\scrU_\p$ is a nonempty open subset of $K_\p^2$ since $U_\p$ is a nonempty open subset of $K_\p$ (when $\p|2$ we also use that $U_\p\cap \OO_\p=\emptyset$ to ensure that $\scrU_\p$ is nonempty). 

Let $\scrU_{q_1}$ be the set of $(\alpha,\beta)\in \OO_{\q_1}^2$ for which $v_{\q_1}(\beta)=0$, $\beta$ is not a square modulo $\q_1$, and $v_{\q_1}(\alpha-e \beta)=1$  for some $e\in \calM$.  Let $\scrU_{q_2}$ be the set of $(\alpha,\beta)\in K_{\q_2}^2$ for which $v_{\q_2}(\beta)=0$, $\beta$ is not a square modulo $\q_2$, and $v_{\q_2}(\alpha-e \beta)=1$  for some $e\in \calM'$.  Since $\calM\cup \calM' \subseteq \OO_K$, we find that $\scrU_{\q_i}$ is a nonempty open subset of $\OO_{\q_i}^2\subseteq K_{\q_i}^2$ for each $i\in \{1,2\}$.

Define $S':=S\cup\{\q_1,\q_2\}$.  By \cite[Proposition~13.2]{Kai} and the openness of our sets, there are nonzero $m,n\in \OO_{K,S'}$ such that the elements $m-en$ with $e\in \calB$ generate distinct prime ideals of $\OO_{K,S'}$, the pair $(m,n)$ lies in $\scrU_\p$ for all $\p\in S'$, and $m/n$ lies in $U_v$ for all $v\in S_\infty$.  Properties (\ref{P:additive a})--(\ref{P:additive d}) all hold by our definitions.  For (\ref{P:additive e}) note that $v_\p(m)=0$ and $v_\p(n)>0$ for all $\p$ dividing $2$ (this uses that $S$ contains all the prime ideals dividing $2$). 

We have $m,n\in \OO_{\q_i}$ for $i\in \{1,2\}$ since $\scrU_{\q_i}\subseteq \OO_{q_i}^2$.  Since $m$ and $n$ are in $\OO_{K,S'}$ we deduce that $m$ and $n$ lie in $\OO_{K,S}$.
\end{proof}

For the rest of the section, we fix $m,n \in \OO_{K,S}$ as in Proposition~\ref{P:additive} and define 
\[
t:=m/n \in K. 
\]
We have $t\notin D$ since $t\in U_v$ for all $v\in S\cup S_\infty$.  Set $S':=S\cup \{\q_1,\q_2\}$. For each $e \in \calB$, there is a unique nonzero prime ideal $\p_e\notin S'$ of $\OO_K$ such that 
\begin{align}\label{E:pe}
(m-en)\OO_{K,S'}=\p_e \OO_{K,S'}.
\end{align}
The prime ideals $\{\p_e\}_{e\in \calB}$ are distinct and do not lie in $S'$ by our choice of $m$ and $n$.

\begin{lemma} \label{L:pe valuation}
For any $e,e'\in \calB$, we have $v_{\p_e}(n)=0$ and 
\[
v_{\p_e}(m-e'n)= 
\begin{cases} 1 & \text{ if $e=e'$},\\ 0 & \text{ if $e\neq e'$.} \end{cases}
\]
\end{lemma}
\begin{proof}
Since the prime ideals $\{\p_e\}_{e\in \calB}$ are distinct and do not lie in $S'$, the desired expression for $v_{\p_e}(m-e'n)$ follows directly from the factorizations (\ref{E:pe}).

With $\p:=\p_e$, it remains to show that $v_{\p}(n)=0$.  Assume to the contrary that $v_{\p}(n)\neq 0$ and hence $n \in \p\OO_p$.  We also have $m\in \p\OO_\p$ since $m-en\in \p\OO_\p$ and $e\in \OO_K$.  Therefore,  $m-e'n \in \p\OO_\p$ for all $e'\in \calB$.  This contradicts that the ideals $(m-e'n)\OO_{K,S'}$ with $e'\in \calB$ are distinct prime ideals.  Therefore, we have $v_\p(n)=0$.
\end{proof}

We will now study the elliptic curves $E_t$ and $E_t'$ with our specific choice of $t$.

\begin{lemma} \label{L:good reduction for Et}
The curves $E_t$ and $E'_t$ over $K$ have good reduction at all nonzero prime ideals $\p\notin S \cup \{\q_1,\q_2\} \cup\{\p_e: e \in \calB\}$ of $\OO_K$.
\end{lemma}
\begin{proof}
Take any  nonzero prime ideal $\p\notin S\cup\{\q_1,\q_2\} \cup\{\p_e: e \in \calB\}$ of $\OO_K$.  If $v_\p(t)<0$, then $E_t$ has good reduction at $\p$ by Lemma~\ref{L:good reduction at infinity}.  So we may assume that $v_\p(t)\geq 0$.  We have $a,b,\Delta\in \OO_K[T]$ so setting $T=t$ in the model (\ref{E:basic eqn}) gives a Weierstrass model for $E_t$ with coefficients in $\OO_\p$ that has discriminant $\Delta(t) \in \OO_\p$.  

It suffices to prove that $\Delta(t)$ lies in $\OO_\p^\times$, so assume to the contrary that $\Delta(t)$ is an element of $\p\OO_\p$.  Since the leading coefficient of $\Delta$ lies in $\OO_{K,S_0}^\times$ and $\calB\subseteq \OO_K$ is the set of roots of $\Delta$ in $\Kbar$, we have $t-e \in \p\OO_\p$ for some $e\in \calB$.  Multiplying by $n\in \OO_{K,S}$ this implies that $m-en \in \p\OO_\p$.   The factorization (\ref{E:pe}) implies that $\p\in S\cup\{\q_1,\q_2\} \cup\{\p_e\}$ which contradicts our choice of $\p$.  Therefore, $\Delta(t)\in\OO_\p^\times$ and hence $E_t$ has good reduction at $\p$.   The elliptic curve $E'_t$ also has good reduction at $\p$ since it is isogenous to $E_t$
\end{proof}

We have a degree $2$ isogeny $\phi_t\colon E_t\to E_t'$ whose kernel is generated by $(0,0)$.  Let $\hat\phi_t\colon E_t'\to E_t$ be the dual isogeny of $\phi_t$.

\begin{lemma} \label{L:final sieve ingredients}
\begin{romanenum}
\item \label{L:final sieve ingredients i}
If $v\in S \cup S_\infty$, then $\Im(\delta_{E'_t,v}) = \Phi_v$.
\item \label{L:final sieve ingredients ii}
If $\p\notin S$ is a nonzero prime ideal of $\OO_K$, then
\[
\tfrac{1}{2} |\Im(\delta_{E'_t,\p})| = 
\begin{cases}
2 & \text{ if $\p\in \{\q_1\}\cup \{\p_e: e\in \calM\}$,} \\
\tfrac{1}{2} & \text{ if $\p\in \{\q_2\}\cup \{\p_e: e\in \calM'\}$,} \\
1 & \text{ otherwise.}
\end{cases}
\]
\item \label{L:final sieve ingredients iii}
We have $|\Sel_{\phi_t}(E_t/K)|/|\Sel_{\hat\phi_t}(E'_t/K)|=2^{|\calM|-|\calM'|}$.
\end{romanenum}
\end{lemma}
\begin{proof}
Part (\ref{L:final sieve ingredients i}) follows from Lemma~\ref{L:balance}(\ref{L:balance a}) and that $t=m/n$ is an element of $K\cap U_v$ for all $v\in S\cup S_\infty$.   Take any nonzero prime ideal $\p\notin S$ of $\OO_K$.

Consider the case where $\p=\p_e$ for some $e\in \calB$.  We have $v_\p(t-e)=v_\p(m-en)-v_\p(n)=1$, where the last equality uses Lemma~\ref{L:pe valuation}. Part (\ref{L:final sieve ingredients ii}) for $\p=\p_e$ follows from Lemma~\ref{L:sieve input}.

By Propostion~\ref{P:additive} and our choice of $t$, we have $v_{\q_i}(t-e)=1$ for some $e\in \calM$ when $i=1$ and some $e\in \calM'$ when $i=2$. Part (\ref{L:final sieve ingredients ii}) for $\p\in \{\q_1,\q_2\}$ follows from Lemma~\ref{L:sieve input}.

For part (\ref{L:final sieve ingredients ii}), it remains to consider the case where $\p \notin S\cup \{\q_1,\q_2\}\cup \{\p_e: e\in \calB\}$.   The elliptic curve $E_t'$ has good reduction at $\p$ by Lemma~\ref{L:good reduction for Et}.  Therefore, $\Im(\delta_{E_t',\p})=\OO_\p^\times/(\OO_\p^\times)^2$ and hence $\tfrac{1}{2}|\Im(\delta_{E_t',\p})|=1$.  

We now prove (\ref{L:final sieve ingredients iii}).  Set $\tau:=|\Sel_{\phi_t}(E_t/K)|/|\Sel_{\hat\phi_t}(E'_t/K)| $.  By Lemma~\ref{L:Selmer ratio}(\ref{L:Selmer ratio i}), we have $\tau= \prod_v \tfrac{1}{2} |\operatorname{Im}(\delta_{E'_t,v})|$,
where the product is over the places $v$ of $K$.  By part (\ref{L:final sieve ingredients i}) and Lemma~\ref{L:balanced product}, we have
\[
\prod_{v\in S\cup S_\infty} \tfrac{1}{2} |\operatorname{Im}(\delta_{E'_t,v})| = \prod_{v\in S\cup S_\infty} \tfrac{1}{2} |\Phi_v|=1.
\]  
Therefore,  $\tau= \prod_{\p \notin S}\tfrac{1}{2} |\operatorname{Im}(\delta_{E'_t,\p})|$, where $\p$ varies over the nonzero prime ideals of $\OO_K$ not in $S$.  From part (\ref{L:final sieve ingredients ii}), we find that $\tau=2^{1+|\calM|} \cdot (\tfrac{1}{2})^{1+|\calM'|}=2^{|\calM|-|\calM'|}$.
\end{proof}

We can now compute $\Sel_{\phi_t}(E_t/K)$ and show that it is as small as possible given what we know about the image of $\delta_{E'_t}$.

\begin{lemma} \label{L:Selmer computation}
We have $\Sel_{\phi_t}(E_t/K)= \delta_{E'_t}(E'_t(K))= \{ f(t)\cdot (K^\times)^2 : f\in \scrF'\}$.
\end{lemma}
\begin{proof}
We have inclusions 
\begin{align} \label{E:Selmer inclusions final}
\Sel_{\phi_t}(E_t/K)\supseteq \delta_{E'_t}(E'_t(K))\supseteq \{ f(t)\cdot (K^\times)^2 : f\in \scrF'\}
\end{align}
from the definition of the Selmer group and Lemma~\ref{L:geometric inclusion}.   Take any $\alpha \in \Sel_{\phi_t}(E_t/K) \subseteq K^\times/(K^\times)^2$.   To prove the lemma, it suffices to show that $\alpha=f(t)\cdot (K^\times)^2$ for some $f\in \scrF'$.  We have $\alpha=c\cdot (K^\times)^2$ for some $c\in K^\times$.   

Take any nonzero prime ideal $\p$ of $\OO_K$ that is not in the set
\[
S'':=S\cup \{\q_1\}\cup\{\p_e: e \in \calA \cup \calM\}.
\]
If $\p\in \{\q_2\}\cup \{\p_e : e\in \calM'\}$, then $\Im(\delta_{E'_t,\p})=1$ by Lemma~\ref{L:final sieve ingredients}(\ref{L:final sieve ingredients ii}).  If $\p\notin  \{\q_1,\q_2\}\cup\{\p_e: e \in \calB\}$, then $E'_t$ has good reduction at $\p$ by Lemma~\ref{L:good reduction for Et} and hence $\Im(\delta_{E'_t,\p})=\OO_\p^\times/(\OO_\p^\times)^2$.   We have $c\cdot (K_\p^\times)^2 \in \Im(\delta_{E'_t,\p})$ since $\alpha \in \Sel_{\phi_t}(E_t/K)$ and thus $v_\p(c)\equiv 0 \pmod{2}$ for all $\p\notin  S''$.   We chose $S_0$ so that $\OO_{K,S_0}$ was a PID and hence $\OO_{K,S''}$ is also a PID.  So after multiplying $c$ by an appropriate square in $K^\times$, we may assume that $c\in \OO_{K,S''}^\times$.  By Lemma~\ref{L:pe valuation} and (\ref{E:pe}), we have $c=c_1 \prod_{e\in \calA\cup \calM} (m-en)^{g_e}$ for unique $g_e\in \ZZ$ and a unique $c_1 \in \OO_{K,S\cup\{\q_1\}}$.  After multiplying $c$ by a square in $K^\times$, we may assume that 
\begin{align}\label{E:c1 and B}
c=c_1 \prod_{e\in B}(m-en)
\end{align}
for some subset $B\subseteq \calA\cup \calM$ and unit $c_1 \in \OO_{K,S\cup\{\q_1\}}^\times$.

We claim that $m-en\in\OO_{K,S}$ is not a square modulo $\q_2$ for all $e\in \calA \cup \calM$.    Take any $e\in \calA\cup \calM$.  By our choice of $m$ and $n$, there is an $e'\in \calM'$ for which $m-e'n \equiv 0 \pmod{\q_2}$.  We have 
\[
m-en=(m-e'n)+(e'-e)n \equiv (e'-e) n \pmod{\q_2}.
\]
By our choice of $S_0$ and $\q_2$, $e'-e$ is a nonzero square modulo $\q_2$.  By our choice of $n$, $n$ is not a square modulo $\q_2$.  Therefore, $m-en$ is not a square modulo $\q_2$ as claimed.

We have $c\cdot (K_{\q_2}^\times)^2 \in \Im(\delta_{E'_t,\q_2})$ since $\alpha \in \Sel_{\phi_t}(E_t/K)$.   Therefore, $c$ is a square in $K_{\q_2}$ by Lemma~\ref{L:final sieve ingredients}(\ref{L:final sieve ingredients ii}).  We have $c_1 \in \OO_{K,S\cup\{\q_1\}}^\times$ and hence $c_1$ is a nonzero square modulo $\q_2$ by our choice of $\q_2$.  By the above claim and (\ref{E:c1 and B}), we find that $c\in \OO_{\q_2}^\times$ is a square modulo $\q_2$ if and only if $|B|$ is even.    Since $c$ is a square in $K_{\q_2}$ we deduce that $|B|$ is even.  We have $n^{|B|} \in (K^\times)^2$ since $|B|$ is even and hence we may assume that $c$ was chosen such that 
\[
c=c_1 \prod_{e\in B}(t-e)
\]
with a subset $B\subseteq \calA\cup \calM$ of even cardinality and a unit $c_1 \in \OO_{K,S\cup\{\q_1\}}^\times$.   Since $|B|$ is even, there is a unique $f\in \scrF'$ such that $f=c_2\prod_{e\in B}(T-e)$ with a constant $c_2\in \OO_{K,S_0}^\times$.   Using the inclusions (\ref{E:Selmer inclusions final}), there is no harm in multiplying $\alpha$ by $f(t) \cdot (K^\times)^2$. 

So without loss of generality, we may assume that $\alpha=c\cdot (K^\times)^2$ with $c\in \OO_{K,S\cup\{\q_1\}}^\times$.  It suffices to show that $\alpha=1$.    Since $\pi$ is a prime in $\OO_{K}$ that generates $\q_1$, we may assume that $c$ was chosen so that $c=u \pi^g$ with $u\in \OO_{K,S}^\times$ and $g\in \{0,1\}$.   Take any $v\in S \cup S_\infty$.    We have $c \cdot (K_v^\times)^2\in \Im(\delta_{E'_t,v})$ since $\alpha \in \Sel_{\phi_t}(E_t/K)$.  Since $t=m/n$ lies in $U_v$ by our choice of $m$ and $n$, Lemma~\ref{L:balance} implies that $c \cdot (K_v^\times)^2\in \Im(\delta_{E'_t,v})= \Phi_v$.  The prime $\pi$ was chosen so that it is a square in $K_v$ and hence $u \cdot (K_v^\times)^2\in \Phi_v$.  Since $u\in \OO_{K,S}^\times$ satisfies $u \cdot (K_v^\times)^2\in \Phi_v$ for all $v\in S\cup S_\infty$, we deduce that $u$ is a square in $\OO_{K,S}^\times$ by the isomorphism of Lemma~\ref{L:balance}(\ref{L:balance b}).   Therefore, $\alpha = \pi^g\cdot (K^\times)^2$ for some $g\in \{0,1\}$.

Fix an $e\in \calM$ satisfying Proposition~\ref{P:additive}(\ref{P:additive c}).  Fix an $e'\in \calM'$ satisfying Proposition~\ref{P:additive}(\ref{P:additive d}).

We claim that $m-e'n$ is not a square modulo $\q_1$.   We have $m-en \equiv 0 \pmod{\q_1}$ and hence $m-e'n = (m-en)+(e-e')n\equiv (e-e') n\pmod{\q_1}$.   Since $e-e'$ lies in $\OO_{K,S_0}^\times$ it is a nonzero square modulo $\q_1$.   By our choice of $m$ and $n$, we know that $n$ is not a square modulo $\q_1$.  Therefore, $m-e'n$ is not a square modulo $\q_1$ as claimed.

We have $v_\p(m-e'n)=0$ for all prime ideals of $\OO_K$ dividing $2$ by Proposition~\ref{P:additive}(\ref{P:additive e}).  From the above claim, we have $v_{\q_1}(m-e'n)=0$.  So (\ref{E:pe}) implies that 
\[
(m-e'n)\OO_K=\p_{e'} \prod_{\p \in S\cup\{\q_2\},\,\p \nmid 2} \p^{f_\p}
\] 
for unique $f_\p\in \ZZ$.   Using second power residue symbols for the field $K$, cf.~\cite[VI \S8]{Neukirch}, we have
\[
\legendre{m-e'n}{\pi}=\legendre{\pi}{m-e'n} = \legendre{\pi}{\p_{e'}} \cdot \prod_{\p\in S\cup \{\q_2\},\, \p\nmid 2} \legendre{\pi}{\p}^{f_\p} = \legendre{\pi}{\p_{e'}},
\]
where the first equality uses the general reciprocity law \cite[VI Theorem 8.3]{Neukirch} and the last equality uses Lemmas~\ref{L:pi existence} and \ref{L:q existence}.   Note that in applying the reciprocity law, we have used Lemma~\ref{L:pi existence} which shows that our $\pi$ is a square in $K_v$ for all places $v$ of $K$ that are infinite or divide $2$.  From the claim above, $m-e'n$ is not a square modulo $\q_1=\pi\OO_K$ and hence $\legendre{\pi}{\p_{e'}}=-1$.  Therefore, $\pi$ is not a square in $K_{\p_{e'}}$.  We have $\Im(\delta_{E'_t,\p_{e'}})=1$ by Lemma~\ref{L:final sieve ingredients}(\ref{L:final sieve ingredients ii}) and hence   $\pi\cdot (K_{\p_{e'}}^\times)^2 \notin \Im(\delta_{E'_t,\p_{e'}})$.  Since $\alpha =\pi^g\cdot (K^\times)^2 \in \Sel_{\phi_t}(E_t/K)$ this implies that $g=0$ and hence $\alpha=1$.
\end{proof}

We can now compute the dimensions of the images of $\delta_{E_t}$ and $\delta_{E'_t}$ from which we will be able to determine the rank of $E_t$.

\begin{lemma} \label{L:almost rank}
\begin{romanenum}
\item \label{L:almost rank i}
We have $\Sel_{\hat\phi_t}(E'_t/K)= \delta_{E_t}(E_t(K))= \{ f(t)\cdot (K^\times)^2 : f\in \scrF\}$.
\item \label{L:almost rank ii}
As vector spaces over $\FF_2$, $\delta_{E_t}(E_t(K))$ and $\delta_{E'_t}(E'_t(K))$ have dimensions $|\calA|+|\calM|-1$ and $|\calA|+|\calM'|-1$, respectively.
\end{romanenum}
\end{lemma}
\begin{proof}
For each subset $B\subseteq \calA\cup \calM$ with even cardinality, let $c_B \in K^\times$ be the unique value for which $f_B:=c_B \prod_{e\in B}(T-e)$ lies in $\scrF$;  each $f\in \scrF$ is of the form $f_B$ for a unique such $B$.  By our choice of $S_0$, we have $c_B\in \OO_{K,S_0}^\times$.   Since $f_B(t)=c_Bn^{-|B|}\prod_{e\in B}(m-en)$, Lemma~\ref{L:pe valuation} implies that $v_{\p_e}(f_B(t))=1$ if $e\in B$ and  $v_{\p_e}(f_B(t))=0$ if $e\in (\calA\cup \calM)-B$. In particular, we can recover any polynomial $f\in \scrF$ from the square class $f(t)\cdot (K^\times)^2$.   This proves that 
\[
|\{f(t)\cdot (K^\times)^2: f \in \scrF\}|=|\scrF|=2^{|\calA|+|\calM|-1}.
\]  
Similarly, we have $|\{f(t)\cdot (K^\times)^2: f \in \scrF'\}|=|\scrF'|=2^{|\calA|+|\calM'|-1}$.

By Lemma~\ref{L:Selmer computation}, we have $|\Sel_{\phi_t}(E_t/K)|= |\delta_{E'_t}(E'_t(K))|=2^{|\calA|+|\calM'|-1}$.  By Lemma~\ref{L:final sieve ingredients}(\ref{L:final sieve ingredients iii}) this implies that $|\Sel_{\hat\phi_t}(E'_t/K)|=2^{|\calA|+|\calM|-1}$.  Now consider the inclusions
\begin{align}\label{E:last Selmer inclusions}
\Sel_{\hat\phi_t}(E'_t/K)\supseteq \delta_{E_t}(E_t(K))\supseteq \{ f(t)\cdot (K^\times)^2 : f\in \scrF\} 
\end{align}
coming from the definition of the Selmer group and Lemma~\ref{L:geometric inclusion}.  We have $|\delta_{E_t}(E_t(K))|=2^{|\calA|+|\calM|-1}$ since we have shown that the other two groups occurring in (\ref{E:last Selmer inclusions}) have this cardinality.   In particular, the inclusions in (\ref{E:last Selmer inclusions}) are actually equality which gives (\ref{L:almost rank i}).   Part (\ref{L:almost rank ii}) is immediate since we have found the cardinality of the image of $\delta_{E_t}$ and the image of $\delta_{E'_t}$.
\end{proof}

By Lemma~\ref{L:almost rank}(\ref{L:almost rank ii}), we have 
\[
\dim_{\FF_2} \delta_{E_t}(E_t(K)) + \dim_{\FF_2} \delta_{E'_t}(E'_t(K)) -2  = (|\calA|+|\calM|-1)+(|\calA|+|\calM'|-1)-2=r.
\]  
Therefore, the elliptic curve $E_t$ over $K$ has rank $r$ by Lemma~\ref{L:rank via quotients}.

\subsection{End of proof}
The following result was proved in \S\ref{SS:choice of elliptic curve}.

\begin{lemma} \label{L:final}
For any finite set $D$ satisfying $\calB\subseteq D\subseteq K$, there is a $t\in K- D$ such that the elliptic curve $E_t$ over $K$ has rank $r$.  \qed
\end{lemma}

Let $R$ be the set of $t\in K-\calB$ for which $E_t$ has rank $r$.   If $R$ is finite, then Lemma~\ref{L:final} with $D:=R\cup\calB$ implies that there is a $t\in K-(R\cup \calB)$ such that $E_t$ has rank $r$ which contradicts the definition of $R$.   Therefore, the set $R$ is infinite.

Let $J\in K(T)$ be the $j$-invariant of the elliptic curve $E$ over $K(T)$; it is nonconstant since by assumption $E$ has a fiber of multiplicative reduction.    For $t\in K-\calB$, the $j$-invariant $j(E_t)$ of the elliptic curve $E_t$ is equal to $J(t)\in K$.   For any $j\in K$, there are only finitely many $t\in K-\calB$ for which $J(t)=j$ since $J$ is nonconstant.  Since $R$ is infinite, so is the set $\{J(t): t\in R\}=\{j(E_t): t\in R\}$.  The theorem follows since an elliptic curve over $K$ is determined up to $\Kbar$-isomorphism by its $j$-invariant.

\section{Proof of Theorem~\ref{T:main}} \label{S:proof of main}

Take any number field $K$.  For each $0\leq r \leq 4$, we will give an elliptic curve $E$ over $K(T)$ that satisfies all the conditions in \S\ref{S:general} with this particular value of $r$.   Theorem~\ref{T:main} will then follow from Theorem~\ref{T:general}.

\subsection{Rank 0 case}

Let $E$ be the elliptic curve over $K(T)$ defined by the Weierstrass equation
\[
y^2=x^3+2x^2+Tx.
\]
With notation as in \S\ref{S:general}, we have an elliptic curve $E'$ over $K(T)$ given by the model
\[
y^2=x^3-4 x^2 -4(T-1) x.
\]
The discriminants of these models of $E$ and $E'$ are  $\Delta=-2^6 T^2 (T-1)$ and $\Delta'=2^{12}T(T-1)^2$, respectively.   Condition (\ref{I:new a}) holds since $\Delta$ factors into linear terms in $K[T]$.  One can check that 
\[
\calA=\{\infty\}, \quad\calM=\{0\}\quad \text{ and }\quad\calM'=\{1\}.
\]
The elliptic curve $E$ has Kodaira symbols $\operatorname{I}_1$ and $\operatorname{III}^*$ at $1$ and $\infty$, respectively. The elliptic curve $E'$ has Kodaira symbols $\operatorname{I}_1$ and $\operatorname{III}^*$ at $0$ and $\infty$, respectively. Therefore, conditions (\ref{I:new b})--(\ref{I:new e}) hold.

Define the point $P_0=(0,0)$ of $E(K(T))$.  Since $\delta_{E}(P_0)=T \cdot (K(T)^\times)^2$ we deduce that $\delta_{E}(E(K(T)))$ has dimension at least $|\calA|+|\calM|-1=1$ over $\FF_2$.  This verifies condition (\ref{I:new f}).
 
Define the point $Q_0=(0,0)$ of $E'(K(T))$.  Since $\delta_{E'}(Q_0)=-(T-1)\cdot (K(T)^\times)^2$ we deduce that $\delta_{E'}(E'(K(T)))$ has dimension at least $|\calA|+|\calM'|-1=1$ over $\FF_2$. This verifies condition (\ref{I:new g}).

Define $r:=2|\calA|+|\calM|+|\calM'|-4=0$.  By Theorem~\ref{T:general}, we deduce that there are infinitely many elliptic curves over $K$, up to isomorphism over $\Kbar$, of rank $r=0$.

\subsection{Rank 1 case}

Let $E$ be the elliptic curve over $K(T)$ defined by the Weierstrass equation
\[
y^2=x^3+T(T-3)x^2+ T x.
\]
With notation as in \S\ref{S:general}, we have an elliptic curve $E'$ over $K(T)$ given by the model
\[
y^2=x^3-2T(T-3) x^2 + T(T-1)^2(T-4) x.
\]
The discriminants of these models of $E$ and $E'$ are  $\Delta=2^4T^3(T-1)^2(T-4)$ and $\Delta'=2^8 T^3 (T-1)^4 (T-4)^2$, respectively.   Condition (\ref{I:new a}) holds since $\Delta$ factors into linear terms in $K[T]$.  One can check that 
\[
\calA=\{0\}, \quad \calM=\{\infty\}\quad \text{ and }\quad \calM'=\{1,4\}.
\]
The elliptic curve $E'$ has Kodaira symbol $\operatorname{I}_3$ at $\infty$.  The elliptic curve $E$ has split multiplicative reduction at $1$ and has Kodaira symbol $\operatorname{I}_1$ at $4$.  The elliptic curves $E$ and $E'$ both have Kodaira symbol $\operatorname{III}$ at $0$. Therefore, conditions (\ref{I:new b})--(\ref{I:new e}) hold.

Define the point $P_0=(0,0)$ of $E(K(T))$.  Since $\delta_{E}(P_0)=T \cdot (K(T)^\times)^2$ we deduce that $\delta_{E}(E(K(T)))$ has dimension at least $|\calA|+|\calM|-1=1$ over $\FF_2$.  This verifies condition (\ref{I:new f}).
 
Define the points $Q_0=(0,0)$ and $Q_1=(T(T-1),2T(T-1))$ of $E'(K(T))$.  Since $\delta_{E'}(Q_0)=T(T-4) \cdot (K(T)^\times)^2$ and $\delta_{E'}(Q_1)=T(T-1)\cdot (K(T)^\times)^2$ we deduce that $\delta_{E'}(E'(K(T)))$ has dimension at least $|\calA|+|\calM'|-1=2$ over $\FF_2$.   This verifies condition (\ref{I:new g}).

Define $r:=2|\calA|+|\calM|+|\calM'|-4=1$.  By Theorem~\ref{T:general}, we deduce that there are infinitely many elliptic curves over $K$, up to isomorphism over $\Kbar$, of rank $r=1$.

\subsection{Rank 2 case}

Let $E$ be the elliptic curve over $K(T)$ defined by the Weierstrass equation
\[
y^2=x^3+10(T+16)x^2+ 9T(T+16) x.
\]
With notation as in \S\ref{S:general}, we have an elliptic curve $E'$ over $K(T)$ given by the model
\[
y^2=x^3-20(T+16) x^2 +64(T+16)(T+25) x.
\]
The discriminants of these models of $E$ and $E'$ are  $\Delta=2^{10} 3^4 T^2 (T+16)^3 (T+25)$ and $\Delta'=2^{20}3^2 T (T+16)^3 (T+25)^2$, respectively.   Condition (\ref{I:new a}) holds since $\Delta$ factors into linear terms in $K[T]$.  One can check that 
\[
\calA=\{-16,\infty\}, \quad\calM=\{0\}\quad\text{ and }\quad\calM'=\{-25\}.
\]
The elliptic curve $E'$ has Kodaira symbol $\operatorname{I}_1$ at $0$. The elliptic curve $E$ has Kodaira symbol $\operatorname{I}_1$ at $-25$. The elliptic curves $E$ and $E'$ both have Kodaira symbol $\operatorname{III}$ at $-16$. The elliptic curves $E$ and $E'$ both have Kodaira symbol $\operatorname{I}_0^*$ at $\infty$ and $c_\infty(E)=c_\infty(E')=4$. Therefore, conditions (\ref{I:new b})--(\ref{I:new e}) hold.

Define the points $P_0=(0,0)$ and $P_1=(-T,4T)$ of $E(K(T))$.  Since $\delta_{E}(P_0)=T(T+16) \cdot (K(T)^\times)^2$ and $\delta_{E}(P_1)=-T\cdot (K(T)^\times)^2$ we deduce that $\delta_{E}(E(K(T)))$ has dimension at least $|\calA|+|\calM|-1=2$ over $\FF_2$. This verifies condition (\ref{I:new f}).
 
Define the points $Q_0=(0,0)$ and $Q_1=(4(T+16),48(T+16))$ of $E'(K(T))$.  Since $\delta_{E'}(Q_0)=(T+16)(T+25) \cdot (K(T)^\times)^2$ and $\delta_{E'}(Q_1)=(T+16)\cdot (K(T)^\times)^2$ we deduce that $\delta_{E'}(E'(K(T)))$ has dimension at least $|\calA|+|\calM'|-1=2$ over $\FF_2$.   This verifies condition (\ref{I:new g}).

Define $r:=2|\calA|+|\calM|+|\calM'|-4=2$.  By Theorem~\ref{T:general}, we deduce that there are infinitely many elliptic curves over $K$, up to isomorphism over $\Kbar$, of rank $r=2$.

\subsection{Rank 3 case}

Let $E$ be the elliptic curve over $K(T)$ defined by the Weierstrass equation
\[
y^2=x^3- (98T^2 - 9725)\cdot x^2+ 7^4 (T^2-2^2)(T^2-11^2) \cdot x.
\]
With notation as in \S\ref{S:general}, we have an elliptic curve $E'$ over $K(T)$ given by the model
\[
y^2=x^3+(196T^2 - 19450)\cdot x^2-2^6 3^2 5^2 7^2(T^2 - (\tfrac{3161}{280})^2)\cdot x.
\]
The discriminants of these models of $E$ and $E'$ are  $\Delta=-2^{10}3^2 5^2 7^{10} (T^2-2^2)^2 (T^2-11^2)^2 (T^2-(\tfrac{3161}{280})^2)$ and $\Delta'=2^{20}3^4 5^4 7^8 (T^2-2^2) (T^2-11^2) (T^2-(\tfrac{3161}{280})^2)^2$, respectively.   Condition (\ref{I:new a}) holds since $\Delta$ factors into linear terms in $K[T]$.  One can check that 
\[
\calA=\emptyset, \quad\calM=\{ \pm 2, \pm 11\}\quad\text{ and }\calM'=\{\pm \tfrac{3161}{280},  \infty\}.  
\]
The elliptic curve $E'$ has Kodaira symbol $\operatorname{I}_1$ at all points in $\calM$. The elliptic curve $E$ has Kodaira symbol $\operatorname{I}_1$ at the points $\pm \tfrac{3161}{280}$.  The elliptic curve $E$ has split multiplicative reduction at $\infty$.  Therefore, conditions (\ref{I:new b})--(\ref{I:new e}) hold.

Define the following points in $E(K(T))$:  
\begin{align*}
P_0 &:=(0,0),\\
P_1 &:=(7^4(T-2)(T+2), 115248 T(T-2)(T+2)),\\
P_2 &:=(7^2(T-2)(T+11), 4263(T-2)(T+11)).
\end{align*}
By considering $\delta_E(P_i)$ with $0\leq i \leq 2$, we find that $\delta_E(E(K(T)))$ contains the subgroup of $K(T)^\times/(K(T)^\times)^2$ generated by the set 
\[
\{(T-2)(T+2)(T-11)(T+11), (T-2)(T+2), (T-2)(T+11) \}.
\]
In particular, $\delta_E(E(K(T)))$ has dimension at least $|\calA|+|\calM|-1=3$ over $\FF_2$ which verifies condition (\ref{I:new f}).

Define the points $Q_0=(0,0)$ and 
\[
Q_1=(630T + 28449/4, 8820T^2 + \tfrac{177093}{2}  T - \tfrac{995715}{8})
\]
of $E'(K(T))$. By considering $\delta_{E'}(Q_0)$ and $\delta_{E'}(Q_1)$, we find that $\delta_{E'}(E'(K(T)))$ contains the subgroup of $K(T)^\times/(K(T)^\times)^2$ generated by the set 
\[
\{-(T-\tfrac{3161}{280})(T-\tfrac{3161}{280}), 70(T + \tfrac{3161}{280})\}.
\]  
In particular, $\delta_{E'}(E'(K(T)))$ has dimension at least $|\calA|+|\calM'|-1=2$ over $\FF_2$ which verifies condition (\ref{I:new g}).

Define $r:=2|\calA|+|\calM|+|\calM'|-4=3$.  By Theorem~\ref{T:general}, we deduce that there are infinitely many elliptic curves over $K$, up to isomorphism over $\Kbar$, of rank $r=3$.

\subsection{Rank 4 case}

Let $E$ be the elliptic curve over $K(T)$ defined by the Weierstrass equation
\[
y^2=x^3-70(T^2-25^2)\cdot x^2+2^4 7^2 (T^2-11^2)(T^2-25^2) \cdot x.
\]
With notation as in \S\ref{S:general}, we have an elliptic curve $E'$ over $K(T)$ given by the model
\[
y^2=x^3+140(T^2-25^2)\cdot x^2+2^2 3^2 7^2 (T^2-25^2)(T^2-39^2) \cdot x.
\]
The discriminants of these models of $E$ and $E'$ are  $\Delta=2^{14}3^2 7^6 (T^2-11^2)^2  (T^2-25^2)^3 (T^2-39^2)$ and $\Delta'=2^{16} 3^4 7^6 (T^2-11^2)(T^2-25^2)^3(t^2-39^2)^2$, respectively.   Condition (\ref{I:new a}) holds since $\Delta$ factors into linear terms in $K[T]$.  One can check that 
\[
\calA=\{\pm 25\},\quad \calM=\{\pm 11\}\quad \text{ and }\quad \calM'=\{\pm 39\}.
\]
The elliptic curve $E'$ has Kodaira symbol $\operatorname{I}_1$ at the points in $\calM$. The elliptic curve $E$ has Kodaira symbol $\operatorname{I}_1$ at the points in $\calM'$. The elliptic curves $E$ and $E'$ both have Kodaira symbol $\operatorname{III}$ at the points in $\calA$. Therefore, conditions (\ref{I:new b})--(\ref{I:new e}) hold.

Define the following points in $E(K(T))$: 
\begin{align*}
P_0 &:=(0,0),\\
P_1 &:=(14(T-11)(T+11), 1176(T-11)(T+11)),\\
P_2 &:=(2(T-11)(T-25), (36T + 780)(T-11)(T-25)).
\end{align*}
By considering $\delta_E(P_i)$ with $0\leq i \leq 2$, we find that $\delta_E(E(K(T)))$ contains the subgroup of $K(T)^\times/(K(T)^\times)^2$ generated by the set 
\[
\{(T-11)(T+11)(T-25)(T+25), 14(T-11)(T+11), 2(T-11)(T-25)\}.
\]
In particular, $\delta_E(E(K(T)))$ has dimension at least $|\calA|+|\calM|-1=3$ over $\FF_2$ which verifies condition (\ref{I:new f}).

Now consider the following points in $E'(K(T))$:
\begin{align*}
Q_0 &:=(0,0),\\
Q_1 &:=( 2(T-25)(T-39), (64T + 1760)(T-25)(T-39) )\\
Q_2 &:=( -14(T-25)(T+25), 4704(T-25)(T+25)).
 \end{align*}
 By considering $\delta_{E'}(Q_i)$ with $0\leq i \leq 2$, we find that $\delta_{E'}(E'(K(T)))$ contains the subgroup of $K(T)^\times/(K(T)^\times)^2$ generated by the set 
 \[
 \{(T-25)(T+25)(T-39)(T+39), 2(T-25)(T-39), -14(T-25)(T+25)\}.
 \]
 In particular, $\delta_{E'}(E'(K(T)))$has dimension at least $|\calA|+|\calM'|-1=3$ over $\FF_2$ which verifies condition (\ref{I:new g}).

Define $r:=2|\calA|+|\calM|+|\calM'|-4=4$.  By Theorem~\ref{T:general}, we deduce that there are infinitely many elliptic curves over $K$, up to isomorphism over $\Kbar$, of rank $r=4$.

\end{document}